\documentclass[11pt, leqno]{amsart}

\usepackage{amsmath}
\usepackage{amsfonts}
\usepackage{amssymb}
\usepackage{amsthm}
\usepackage{enumitem}
\usepackage{thm-restate}
\usepackage{mathtools}

\usepackage{hyperref}
\hypersetup{hidelinks}

\usepackage{geometry}
\newtheorem{theorem}{Theorem}[section]
\newtheorem{lemma}{Lemma}[section]

\theoremstyle{remark}
\newtheorem*{remark}{Remark}

\theoremstyle{definition}
\newtheorem{definition}[lemma]{Definition}

\numberwithin{equation}{section}

\DeclareMathOperator{\supp}{supp}
\newcommand{\Mod}[1]{\ (\mathrm{mod}\ #1)}

\makeatletter
\@namedef{subjclassname@2020}{\textup{2020} Mathematics Subject Classification}
\makeatother

\title{The Second Moment of Sums of Hecke Eigenvalues II}
\author{Ned Carmichael}
\date{March 5, 2026}
\address{Department of Mathematics, King’s College London, London, WC2R 2LS, UK}
\email{ned.carmichael@kcl.ac.uk}

\begin{document}

\begin{abstract}
Let \(f\) be a holomorphic Hecke cusp form of weight \(k\) for \(\mathrm{SL}_2(\mathbb{Z})\), and let \((\lambda_f(n))_{n\geq 1}\) denote its sequence of normalised Hecke eigenvalues. 
We compute the first and second moments of the sums \(\mathcal S(x,f)=\sum_{x\leq n\leq 2x} \lambda_f(n)\), on average over forms \(f\) of large weight \(k\). 
In the range \(k^2/(8\pi^2)\leq x\leq k^{12/5-\epsilon}\), the size of the second moment lies between \(x^{1/2-o(1)}\) and \(x^{1/2}\). 
This is in sharp contrast to the regime \(x\leq k^{2-o(1)}\), where the second moment was shown in preceding work \cite{paper1} to be of size \(\asymp x\).
\end{abstract}


\maketitle
\section{Introduction}

Let \(k\) be an even positive integer, and let \(\mathcal B_k\) denote an orthogonal basis (with respect to the Petersson inner product) of Hecke eigenforms for the space of weight \(k\) holomorphic cusp forms for \(\mathrm{SL}_2(\mathbb{Z})\). 
In this article, we normalise the Fourier expansion of \(f\in\mathcal B_k\) by 
\[f(z)=\sum_{n\geq1}\lambda_f(n) n^{(k-1)/2} e(nz)\;\;\;\; (z\in\mathbb{H}),\]
and \(\lambda_f(1)=1\).
We study the sums of Hecke eigenvalues
\[\mathcal S(x,f)\coloneqq \sum_{x\leq n\leq 2x} \lambda_f(n),\]
as \(f\) traverses \(\mathcal B_k\) with \(k\to\infty\).  

Similar problems have already been considered. For example, Lester and Yesha \cite{lesteryesha} study the distribution of sums of Hecke eigenvalues over short intervals. Sums of eigenvalues in progressions have also been investigated in \cite{blomer}, \cite{fgkm}, \cite{lauzhao}, \cite{lesteryesha} and \cite{lu}. Notably, Lau and Zhao \cite{lauzhao} prove asymptotics for the variance of these sums (on average over the congruence classes) which demonstrate an interesting transition in the average size of the sums as the length of the sums varies relative to the modulus.

The sums \(\mathcal S(x,f)\) themselves have been studied previously, for fixed \(f\) (and \(k\)). Indeed, Hafner and Ivi\'c \cite[Theorem 1]{hafnerivic} showed
\begin{equation}\label{hafivicbound} 
\mathcal S(x,f)\ll_f x^{1/3}.
\end{equation} 
(This can be slightly improved, see \cite{rankin90}, \cite{wu} and \cite{tangwu}.) Moreover, the following mean-square estimate is known:
\begin{equation}\label{cnmeansquare}
\frac1X\int_0^X \Big|\sum_{n\leq x}\lambda_f(n)\Big|^2dx=c_f X^{1/2}+\mathcal O(\log^2 X).
\end{equation}
This is easily deduced from a result of Chandrasekharan and Narasimhan \cite[Theorem 1]{c-n} (see also \cite{walfisz}). From (\ref{cnmeansquare}), one may expect the sums \(\mathcal S(x,f)\) to be of size roughly \(x^{1/4}\) on average.

%None of the above results are uniform in the weight, however. Deligne's bound states \(|\lambda_f(n)|\leq d(n)\) (where \(d(n)\) denotes the divisor function), and therefore shows \(\mathcal S(x,f)\ll x\log x\) uniformly in \(f\). One can also easily derive uniform bounds using Perron's formula and standard properties of the \(L\)-function \(L(s,f)=\sum_{n\geq1}\lambda_f(n)n^{-s}\). This yields \(\mathcal S(x,f)\ll k^{1+\epsilon}+ x^{1/2+\epsilon}\), which beats the bound \(x\log x\) if \(x\geq k^{1+\epsilon}\). Moreover, if one assumes GRH for \(L(s,f)\) then one obtains \(\mathcal S(x,f)\ll x^{1/2+\epsilon}k^\epsilon\), improving the unconditional bound when \(x\leq k^{2-\epsilon}\).

However, preceding work \cite{paper1} investigated the sums \(\mathcal S(x,f)\) in the regime \(x\leq k^{2-o(1)}\), and observed rather different behaviour.
We review this now.
Define the averaging operator 
\begin{equation}\label{favs} 
\langle g(f)\rangle=\sum_{f\in \mathcal B_k}\omega(f)g(f),
\end{equation}
where \(\omega(f)\) is the \emph{harmonic weight}
\[\omega(f)=\frac{\Gamma(k-1)}{(4\pi)^{k-1}\|f\|^2}.\]
The harmonic weight arises naturally from the Petersson trace formula. We remark that \(\langle 1\rangle =\sum_f \omega(f)=1+\mathcal O(e^{-k})\). In \cite{paper1}, the first and second moments of the sums \(\mathcal S(x,f)\) were studied. It was proved that for \(x,k\to\infty\) with \(x=o(k^2/\log^6k)\), one has 
\begin{equation*}
\langle \mathcal S(x,f)\rangle \ll e^{-\sqrt k} \: \text{ and } \: \langle \mathcal S(x,f)^2\rangle\sim c(x) x,
\end{equation*}
where \(c(x)=c_k(x)\) is an explicit function satisfying \(1/100\leq c(x)\leq 2\). In fact, \(c(x)=1\) provided \(x\notin [k/(8\pi), k/(4\pi)]\). Therefore one expects \(\mathcal S(x,f)\) to be roughly of size \(x^{1/2}\) for most \(x=o(k^2/\log^6 k)\).
%and thus the GRH bound \(\mathcal S(x,f)\ll x^{1/2+\epsilon}k^\epsilon\) is sharp (at least, up to the factor of \((xk)^\epsilon\)) in this range of \(x\).

However, it was shown that the sums \(\mathcal S(x,f)\) transition in size approximately when \(k^2/(32\pi^2)\leq x\leq k^2/(16\pi^2)\), and we expect the sums to be considerably smaller after this transition (i.e. for \(x\geq k^2/(16\pi^2)\)). 
The precise asymptotic behaviour is unclear around the transition, and it is an interesting problem to extend our results to address this regime. 
In this paper, we consider the regime where \(x\geq k^2/(8\pi^2)\), and demonstrate that the average size of \(\mathcal S(x,f)\) is around \(x^{1/4}\) when \(k^2/(8\pi^2)\leq x\leq k^{12/5-\epsilon}\).

\subsection{Statement of Results}

\begin{remark} 
Throughout this article, the implied constants in \(\ll\) and \(\mathcal O\) notation depend at most upon \(\epsilon>0\) unless otherwise specified.
\end{remark}

We now state the theorems. 
The following result is essentially the bound (\ref{hafivicbound}), but holds uniformly for \(f\in\mathcal B_k\). 

\begin{theorem}\label{thmbound}
Let \(f\in \mathcal B_k\). Then for \(x\geq k^2/(8\pi^2)\) and any \(\epsilon>0\), we have
\[\mathcal S(x,f)\ll x^{1/3+\epsilon}.\]
\end{theorem}

We next evaluate the first and second moments of the sums \(\mathcal S(x,f)\). First, we must introduce some notation. Throughout this paper, we write \(\kappa=k-1\) for convenience.  We define
\[\omega(z)=\omega_{\kappa}(z)=(z^2-\kappa^2)^{1/2}-\kappa\arctan\big((z^2/\kappa^2-1)^{1/2}\big)-\pi/4,\]
and set
\begin{multline}\label{bigomegadef}
\Omega(n,x)=\Omega_\kappa(n,x)\\ 
\coloneqq 2(32\pi^2-\kappa^2/(nx))^{-3/4}\sin \omega(4\pi\sqrt{2nx})-(16\pi^2 -\kappa^2/(nx))^{-3/4}\sin\omega(4\pi\sqrt{nx}),
\end{multline}
provided \(nx>\kappa^2/(16\pi^2)\).
Importantly, \(\Omega\) satisfies \(\Omega(n,x)\ll 1\) for all integers \(n\geq1\) whenever \(x\geq k^2/(8\pi^2)\). 
We have the following estimates for the first and second moments.

\begin{theorem}\label{thmmean}
Let \(\epsilon>0\). If \(k^2/(8\pi^2)\leq x\leq k^{4}\), then 
\[\langle \mathcal S(x,f)\rangle= (-1)^{k/2}4\sqrt{2\pi} \Omega(1,x)x^{1/4}+\mathcal O(x^{1/2}k^{-1+\epsilon}).\]
\end{theorem}

\begin{restatable}{theorem}{thmvar}\label{thmvar}
Let \(\epsilon>0\). If \(k^2/(8\pi^2)\leq x\leq k^{12/5}\), then
\begin{equation*}
\langle \mathcal S(x,f)^2\rangle = 32\pi x^{1/2}\sum_{n\geq1} \frac{\Omega(n,x)^2}{n^{3/2}} +\mathcal O(x^{3/4}k^{-3/5+\epsilon})+\mathcal O(k^{29/30+\epsilon}).
\end{equation*}
Moreover, the above main term satisfies
\[x^{1/2}\exp\Big(-\frac{\log x}{\log\log x}\Big) \ll 32\pi x^{1/2}\sum_{n\geq1}\frac{\Omega(n,x)^2}{n^{3/2}} \ll x^{1/2}.\]
\end{restatable}

\begin{remark}
The asymptotic behaviour of the variance
\[\langle (\mathcal S(x,f)-\langle \mathcal S(x,f)\rangle)^2 \rangle=\langle \mathcal S(x,f)^2\rangle -\langle \mathcal S(x,f)\rangle ^2(1+\mathcal O(e^{-k}))\]
is easily deduced from the above. Indeed, if \(k^2/(8\pi^2)\leq x\leq k^{12/5}\), one combines Theorems \ref{thmmean} and \ref{thmvar} to see
\[\mathcal \langle (\mathcal S(x,f)-\langle \mathcal S(x,f)\rangle)^2 \rangle=32\pi x^{1/2}\sum_{n\geq2} \frac{\Omega(n,x)^2}{n^{3/2}} +\mathcal O(x^{3/4}k^{-3/5+\epsilon})+\mathcal O(k^{29/30+\epsilon}).\]
As before, the main term satisfies
\[x^{1/2}\exp\Big(-\frac{\log x}{\log\log x}\Big)\ll 32\pi x^{1/2}\sum_{n\geq2} \frac{\Omega(n,x)^2}{n^{3/2}}\ll x^{1/2}.\]
\end{remark}

We now offer an explanation for the significant transition in the sizes of the sums \(\mathcal S(x,f)\) (occurring approximately when \(k^2/(32\pi^2)\leq x\leq k^2/(16\pi^2)\)). This phenomenon may be explained with reference to the Vorono\"i type summation formula for \(\mathcal S(x,f)\). This is formulated precisely in Lemma \ref{vorprop}, but very roughly speaking it states that for \(f\in\mathcal B_k\), 
\begin{equation}\label{roughvor}
\mathcal S(x,f)\approx 2\pi (-1)^{k/2} x\sum_{n\geq1}\lambda_f(n)\int_1^2 J_{k-1}(4\pi\sqrt{nxt})dt.
\end{equation}
Here \(J_{k-1}\) denotes the Bessel function. The behaviour of \(J_{k-1}(z)\) is roughly as follows. When the argument \(z\) is much smaller than the index \(k-1\), \(J_{k-1}(z)\) is negligibly small. 
When \(z\approx k-1\), \(J_{k-1}(z)\) increases to a large global maximum.
To the right of this peak, once \(z\) is larger than \(k-1\), the Bessel function oscillates with decaying amplitude. 

Note that if \(x\leq k^2/(32\pi^2)\) then there exist integers \(n\geq1\) with \( k^2/(32\pi^2x)\leq n\leq k^2/(16\pi^2x)\). For these \(n\)'s, there exists \(t\in[1,2]\) such that \(4\pi\sqrt{nxt}\approx k\), and therefore one expects the peak of the Bessel function \(J_{k-1}(4\pi\sqrt{nxt})\) to produce a large contribution to the integral part of (\ref{roughvor}). Consequently, we expect that the sums \(\mathcal S(x,f)\) can be large when \(x\leq k^2/(32\pi^2)\). 

On the other hand, if \(x\geq k^2/(16\pi^2)\), then \(4\pi\sqrt{nxt}\geq k\) for all \(n\geq 1\) and \(1\leq t\leq 2\). Thus all Bessel functions appearing in the integral part of (\ref{roughvor}) are in their oscillatory regime, and consequently one expects a lot of cancellation in these integrals. Consequently, once \(x\geq k^2/(16\pi^2)\), we expect the sizes of the sums \(\mathcal S(x,f)\) to be relatively small.

\subsection{Acknowledgements} 

I would like to thank Stephen Lester for many helpful discussions and useful comments on an earlier draft of this paper. I also thank Bingrong Huang for pointing out the relevant work \cite{hough}. This work was supported by the Additional Funding Programme for Mathematical Sciences, delivered by EPSRC (EP/V521917/1) and the Heilbronn Institute for Mathematical Research.

\section{Preliminaries}\label{secprelim}

In this section we summarise some results for later use.

\subsection{Bessel Functions}

The Bessel functions \(J_\nu(z)\) will appear throughout this paper. They are defined \cite[p.40]{watson} by
\begin{equation}\label{besseldef}
J_\nu(z)=\sum_{l\geq 0}\frac{(-1)^l}{l!\Gamma(\nu+1+l)}\Big(\frac{z}{2}\Big)^{\nu+2l}.
\end{equation} 
Throughout this section, we assume that the index \(\nu\) and the argument \(z\) are real and positive. The behaviour of \(J_\nu(z)\) is roughly as follows. When the argument \(z\) is small relative to the index \(\nu\), \(J_\nu(z)\) is small. The transition regime is where \(|z- \nu|\ll \nu^{1/3}\), i.e. the argument is approximately equal to the index. In this regime \(J_\nu(z)\) reaches its global maximum of size \(\approx \nu^{-1/3}\). Finally, once the argument becomes larger than the index, the Bessel functions oscillate, with the amplitude of the oscillations decaying roughly like \(z^{-1/2}\).

A detailed discussion of the Bessel function (with references) is given in \cite[Appendix A]{paper1}. Here it suffices to summarise some useful facts in the following lemma.

\begin{lemma}\label{bessellem}
Let \(\nu>0\) be sufficiently large. We have the following.
\begin{enumerate}[label=(\roman*)]

\item \label{besi} If \(0\leq z\leq (\nu+1)/4\), then
\begin{equation}\label{bessboundsmallarg}
J_\nu(z)\ll z^2\exp\Big(-\frac{14\nu}{13}\Big).
\end{equation}

\item \label{besii} If \(0\leq z\leq (\nu+1)-(\nu+1)^{1/3+\delta}\) for some \(0\leq\delta\leq 2/3\), then
\begin{equation}\label{bessel1}
J_\nu(z)\ll \exp(-\nu^\delta).
\end{equation}

\item \label{besiii} Uniformly for \(z\geq0\), we have the bound 
\begin{equation}\label{bessel2}
J_\nu(z)\ll \nu^{-1/3}.
\end{equation}

\item \label{besiv} If \(z\geq \nu+\nu^{1/3}\), then 
\begin{multline}\label{bessel3}
J_\nu(z)=\sqrt{\frac{2}{\pi}}(z^2-\nu^2)^{-1/4}\cos \omega(z)+\sqrt{\frac{2}{\pi}}(z^2-\nu^2)^{-3/4}\Big(\frac{1}{8}+\frac{5}{24}\frac{\nu^2}{z^2-\nu^2}\Big)\sin \omega(z)\\
+\mathcal O\Big(\frac{z^4}{(z^2-\nu^2)^{13/4}}\Big),
\end{multline}
where \(\omega\) is given by 
\begin{equation}\label{omega}
\omega(z)=\omega_\nu(z)=(z^2-\nu^2)^{1/2}-\nu\arctan\big((z^2/\nu^2-1)^{1/2}\big)-\pi/4.
\end{equation}
\end{enumerate}
\end{lemma}

\begin{remark}
For convenience, we record
\begin{equation}\label{omega'}
\omega'(z)=\frac{(z^2-\nu^2)^{1/2}}{z},
\end{equation}
and
\begin{equation}\label{omega''}
\omega''(z)=\frac{\nu^2}{z^2(z^2-\nu^2)^{1/2}}.
\end{equation}
\end{remark}

\begin{proof}
The bounds given in parts \ref{besi}, \ref{besii} and \ref{besiii} may be found in \cite[Lemma A.2]{paper1}. We now address part \ref{besiv}. The full asymptotic expansion in the oscillatory regime is computed, for example, by Olver \cite{olver} -- see Ch.10, §8, ex.8.2 (cf. Ch.10, §7--8). Taking \(n=1\) in Olver's expansion, i.e. truncating after the first two terms gives that for \(z>\nu\),
\begin{multline}\label{watsonlargearg}
J_\nu(z)=\sqrt{\frac2\pi}(z^2-\nu^2)^{-1/4}\Big\{\cos\omega(z)\hat U_0\Big(\Big(\frac{z^2}{\nu^2}-1\Big)^{-1/2}\Big)+\frac{\sin\omega(z)}{\nu}\hat U_1\Big(\Big(\frac{z^2}{\nu^2}-1\Big)^{-1/2}\Big)\\
+\mathcal O\Big(\nu^{-2}\exp\Big(\frac2\nu \mathrm{Var}\Big(\hat U_1; 0, \Big(\frac{z^2}{\nu^2}-1\Big)^{-1/2}\Big)\Big) \mathrm{Var}\Big(\hat U_2; 0, \Big(\frac{z^2}{\nu^2}-1\Big)^{-1/2}\Big)\Big)\Big\}.
\end{multline}
Here \(\mathrm{Var}(f;a,b)\) denotes the total variation of a function \(f\) on the interval \([a,b]\), and the \(\hat U_i\) are polynomials (cf. Ch.10, §7, (7.11)) given by
\begin{equation*}
\Hat{U}_0(t)=1, \Hat{U}_1(t)=\frac{1}{24} (3t+5t^3) \text{ and } \Hat{U}_2(t)=\frac{1}{1152} (81t^2+462t^4+385t^6).
\end{equation*}
Since \(\hat U_1\) and \(\hat U_2\) are both increasing and \(\hat U_1(0)=\hat U_2(0)=0\), one has 
\begin{equation*} 
\mathrm{Var}\Big(\hat U_1; 0, \Big(\frac{z^2}{\nu^2}-1\Big)^{-1/2}\Big)=\hat U_1\Big(\Big(\frac{z^2}{\nu^2}-1\Big)^{-1/2}\Big)
=\frac{\nu}{8(z^2-\nu^2)^{1/2}}+\frac{5\nu^3}{24(z^2-\nu^2)^{3/2}}.
\end{equation*}
Our assumption \(z\geq \nu+\nu^{1/3}\implies z^2-\nu^2\gg \nu^{4/3}\) implies
\begin{equation*} 
\mathrm{Var}\Big(\hat U_1; 0, \Big(\frac{z^2}{\nu^2}-1\Big)^{1/2}\Big)\ll \nu^{1/3}+\nu\ll \nu.
\end{equation*}
Similarly,
\begin{equation*}
\mathrm{Var}\Big(\hat U_2; 0, \Big(\frac{z^2}{\nu^2}-1\Big)^{-1/2}\Big)\Big)
=
\hat U_2\Big(\Big(\frac{z^2}{\nu^2}-1\Big)^{-1/2}\Big)
\ll 
\frac{\nu^2}{(z^2-\nu^2)}\Big(1+\frac{\nu^4}{(z^2-\nu^2)^2}\Big).
\end{equation*}

%\begin{multline*}
%\mathrm{Var}\Big(\hat U_2; 0, \Big(\frac{z^2}{\nu^2}-1\Big)^{-1/2}\Big)\Big)=\hat U_2\Big(\Big(\frac{z^2}{\nu^2}-1\Big)^{-1/2}\Big)\\
%\ll \Big(\frac{z^2}{\nu^2}-1\Big)^{-1}+ \Big(\frac{z^2}{\nu^2}-1\Big)^{-2}+ \Big(\frac{z^2}{\nu^2}-1\Big)^{-3}\ll \frac{\nu^2}{(z^2-\nu^2)}\Big(1+\frac{\nu^4}{(z^2-\nu^2)^2}\Big).
%\end{multline*}
So the error term in (\ref{watsonlargearg}) is 
\begin{multline}\label{watsonlargeargerrorreal}
\nu^{-2}\exp\Big(\frac2\nu \mathrm{Var}\Big(\hat U_1; 0, \Big(\frac{z^2}{\nu^2}-1\Big)^{-1/2}\Big)\Big) \mathrm{Var}\Big(\hat U_2; 0, \Big(\frac{z^2}{\nu^2}-1\Big)^{-1/2}\Big)\\
\ll \exp(\mathcal O(1))\frac{1}{(z^2-\nu^2)}\Big(1+\frac{\nu^4}{(z^2-\nu^2)^2}\Big)\ll \frac{z^4}{(z^2-\nu^2)^3}.
\end{multline}
(The last inequality holds for any \(z>\nu\).) 
%In the last step, we used that for any \(z>\nu\),
%\[1\ll \frac{z^4}{(z^2-\nu^2)^2} \text{ and } \frac{\nu^4}{(z^2-\nu^2)^2}\ll \frac{z^4}{(z^2-\nu^2)^2}.\] 
The asymptotic given in part \ref{besiv} now follows immediately from (\ref{watsonlargearg}) and the bound (\ref{watsonlargeargerrorreal}).
\end{proof}

\subsection{The Petersson Trace Formula}

The key tool for computing averages is the Petersson trace formula, which we now formulate. This gives an expression for the averages \(\langle \lambda_f(n)\lambda_f(m)\rangle\) as diagonal and off-diagonal terms. Conveniently, the off-diagonal is negligibly small in certain ranges of \(m\) and \(n\).

\begin{lemma}[Petersson Trace Formula]\label{trace}
Let \(S(m,n;c)\) denote the Kloosterman sums 
\[S(m,n;c)=\sum_{\substack{a\Mod c\\ (a,c)=1}} e\Big(\frac{a^*m+an}{c}\Big),\]
where \(a^*\) denotes the multiplicative inverse of \(a\) modulo \(c\). Set
\[\delta_{mn}=\begin{cases} 1 &\text{ if } m=n,\\ 0 &\text{ otherwise}.\end{cases}\]
Let \(k\) be sufficiently large, and \(\langle\cdot\rangle\) as given in (\ref{favs}). We have the following.
\begin{enumerate}[label=(\roman*)]
\item \label{tracei}\cite[Theorem 3.6]{iwaniec} For any positive integers \(m\) and \(n\),
\begin{equation*}\label{pet}
\langle \lambda_f(n)\lambda_f(m)\rangle=\delta_{mn}+2\pi (-1)^{k/2}\sum_{c\geq1} c^{-1}S(m,n;c)J_{k-1}\Big(\frac{4\pi\sqrt{mn}}{c}\Big).
\end{equation*}
\item \label{traceii} \cite[Lemma 2.1]{rudnicksound} If \(m\) and \(n\) are positive integers satisfying \(mn\leq k^2/10^4\), then
\begin{equation*}
\langle \lambda_f(n)\lambda_f(m)\rangle=\delta_{mn}+\mathcal O(e^{-k}).
\end{equation*}
\end{enumerate}
\end{lemma}

\subsection{A Vorono\"i Summation Formula}

We now state a Vorono\"i type summation formula for the sums \(\mathcal S(x,f)\). A version of  this for smoothed sums of Hecke eigenvalues is given in \cite[ex.9, p.83]{iwanieckowalski}; the version here is adapted for the sharp cut-off sums \(\mathcal S(x,f)\). One first requires a suitable smoothing function, given as follows.

\begin{definition}\label{wdef}
Given \(\Delta\geq1\), denote by \(w=w_\Delta\) a smooth function \(w:\mathbb{R}\to \mathbb{R}\) satisfying the following:
\begin{itemize}
\item \(\supp w= [1,2],\)
\item \(w(\xi)=1\) for \(1+\Delta^{-1}\leq \xi\leq 2-\Delta^{-1},\)
\item for all integers \(j\geq0\) and all \(\xi\), we have \(w^{(j)}(\xi)\ll_j \Delta^j\).
\end{itemize}
\end{definition}

\begin{lemma}[Vorono\"i Summation Formula {\cite[§2.3]{paper1}}] \label{vorprop}
Let \(\Delta\) be a sufficiently large parameter satisfying \(\Delta\leq x^{1-\epsilon}\) for some \(\epsilon>0\). Let \(w=w_\Delta\) be the associated smooth function given in Definition \ref{wdef} above. Then for \(f\in\mathcal B_k\), we have
\begin{equation*}
\mathcal S(x,f)=2\pi (-1)^{k/2} x\sum_{n\geq1} \lambda_f(n)\tilde w\Big(\frac{nx}{k^2+\Delta^2}\Big)+\mathcal O \Big(\frac{x\log x}{\Delta}\Big),
\end{equation*}
where 
\begin{equation}\label{tildew}
\tilde{w}(\xi)=\tilde w_\Delta(\xi)=\int_0^\infty w(t)J_{k-1}(4\pi \sqrt{(k^2+\Delta^2)\xi t})dt.
\end{equation}
Moreover, for any integer \(A\geq0\), \(\tilde w\) satisfies the bound 
\begin{equation} \label{tildewibped} 
\tilde w (\xi)\ll_A \xi^{-A}.
\end{equation}
\end{lemma}

Throughout this paper, we will require several estimates for the weights (\ref{tildew}). Observe
\begin{equation}\label{tildewbutforn}
\tilde{w}\Big(\frac{nx}{k^2+\Delta^2}\Big)=\int_0^\infty w(t)J_{k-1}(4\pi \sqrt{nxt})dt. 
\end{equation}
Under our assumption \(x\geq k^2/(8\pi^2)\), the Bessel function \(J_{k-1}(4\pi\sqrt{nxt})\) above is always in its oscillatory regime. This leads to cancellation in (\ref{tildewbutforn}) (as discussed in the introduction). By applying asymptotics for the Bessel function valid in the oscillatory regime, we capture this cancellation in the following lemma.

\begin{lemma}\label{wxbig}
For \(\Delta\geq1\), let \(w=w_\Delta\) be given as in Definition \ref{wdef} and \(\tilde w=\tilde w_\Delta\) as in (\ref{tildew}). Suppose \(x\geq k^2/(8\pi^2)\). Then for \(n\geq1\), 
\begin{equation*}
\tilde w\Big(\frac{nx}{k^2+\Delta^2}\Big)=\sqrt{\frac{2}{\pi}}\int_0^\infty w(t)(16\pi^2 nxt-\kappa^2)^{-1/4}\cos \omega(4\pi \sqrt{nxt})dt\\
+\mathcal O((nx)^{-5/4}).
\end{equation*}
\end{lemma}

\begin{proof} We will apply the asymptotic (\ref{bessel3}) of Lemma \ref{bessellem} and integrate by parts. Indeed, since we assume \(x\geq k^2/(8\pi^2)\) it is clear that \(4\pi \sqrt{nxt}\geq \sqrt2 k\) for all \(n\geq1\) and \(t\in\supp w=[1,2]\). Thus (\ref{bessel3}) gives an asymptotic expansion for \(J_{k-1}(4\pi\sqrt{nxt})\) valid in this range of \(x\). Replacing this in (\ref{tildewbutforn}) (and noting \(16\pi^2 nxt\geq 2k^2\implies(16\pi^2nxt-\kappa^2)\asymp nx\)) shows
\begin{multline}\label{weight0}
\tilde w\Big(\frac{nx}{k^2+\Delta^2}\Big)=\sqrt{\frac{2}{\pi}}\int_0^\infty w(t)(16\pi^2 nxt-\kappa^2)^{-1/4}\cos \omega(4\pi \sqrt{nxt})dt\\
+\sqrt{\frac{2}{\pi}}\int_0^\infty w(t)(16\pi^2 nxt-\kappa^2)^{-3/4}\Big(\frac{1}{8}+\frac{5}{24}\frac{\kappa^2}{(16\pi^2 nxt-\kappa^2)}\Big)\sin \omega(4\pi \sqrt{nxt})dt\\
+\mathcal O((nx)^{-5/4}).
\end{multline}

We now integrate by parts in the latter integral above. First note that (\ref{omega'}) implies
\begin{equation}\label{omega'again} \frac{\partial}{\partial t} \omega(4\pi \sqrt{nxt})=\frac{1}{2t}(16\pi^2 nxt-\kappa^2)^{1/2}.
\end{equation}
This shows the second integral in (\ref{weight0}) is
\begin{align*}
&\int_0^\infty w(t)(16\pi^2 nxt-\kappa^2)^{-3/4}\Big(\frac{1}{8}+\frac{5}{24}\frac{\kappa^2}{(16\pi^2 nxt-\kappa^2)}\Big)\sin \omega(4\pi \sqrt{nxt})dt\\
=&-2\int_0^\infty tw(t)(16\pi^2 nxt-\kappa^2)^{-5/4}\Big(\frac{1}{8}+\frac{5}{24}\frac{\kappa^2}{(16\pi^2 nxt-\kappa^2)}\Big) d\big(\cos \omega(4\pi \sqrt{nxt})\big)\\
=&2\int_0^\infty \Big\{\Big(w(t)+tw'(t)-tw(t)\frac{5}{4}\frac{16\pi^2 nx}{(16\pi^2nxt-\kappa^2)}\Big)\Big(\frac{1}{8}+\frac{5}{24}\frac{\kappa^2}{(16\pi^2 nxt-\kappa^2)}\Big)\\
&\hspace{4em}-tw(t)\frac{5}{24}\frac{16\pi^2nx\kappa^2}{(16\pi^2nxt-\kappa^2)^2}\Big\}(16\pi^2nxt-\kappa^2)^{-5/4} \cos \omega(4\pi \sqrt{nxt})dt\ll (nx)^{-5/4}.
\end{align*}
The final bound above follows from the bounds \(w\ll 1\), \(w'\ll \Delta\) and the observation \(\supp w'\subset [1,1+\Delta^{-1}]\cup[2-\Delta^{-1}, 2]\) (which has measure \(2\Delta^{-1}\)) (see Definition \ref{wdef}). The lemma now follows immediately from (\ref{weight0}).
\end{proof}

\subsection{A Smooth Partition of Unity}

We construct a smooth partition of unity, for use later on.
This is standard, and very similar to (for example) \cite[Lemme 2]{fouvry}.

\begin{lemma}[Smooth Partition of Unity]\label{smthpart}
Let \(\alpha\in\mathbb{R}\) and \(L>0\). There exists a sequence of real-valued smooth functions \((b_l^{L,\alpha})_{l\in\mathbb{Z}}\) satisfying the following.
\begin{enumerate}[label=(\roman*)]
    \item \label{unityi} For any \(\xi\), \(\sum_{l\in\mathbb{Z}}b_l^{L, \alpha}(\xi)=1\).
    \item \label{unityii} For any \(l\geq 1\) we have \(\supp b_l^{L,\alpha}=[\alpha+2^{l-1}L,\alpha+2^{l+1}L]\), and for any \(l\leq -1\) we have \(\supp b_l^{L,\alpha}=[\alpha-2^{|l|+1}L,\alpha-2^{|l|-1}L]\). Finally, for \(l=0\) we have \(\supp b_0^{L,\alpha}=[\alpha-2L,\alpha+2L]\).
    \item \label{unityiii} For a non-negative integer \(j\) and any \(l\in\mathbb{Z}\), we have \(\frac{d^j}{d\xi^j}b_l^{L,\alpha}(\xi)\ll_j 2^{-j|l|}L^{-j}\).
\end{enumerate}
\end{lemma}

\begin{proof}
One can construct a smooth transition function \(h:[0,1]\to\mathbb{R}\) satisfying \(h(0)=0\), \(h(1)=1\), and \(\lim_{\xi\to0^+}h^{(j)}(\xi)=\lim_{\xi\to1^-}h^{(j)}(\xi)=0\) for any \(j\geq0\). Equipped with such a \(h\), define
\[b_0^{L,\alpha}(\xi)=\begin{cases} 0&\text{ if } \xi\leq \alpha-2L \\ 1-h\Big(\frac{\alpha-\xi}{L}-1\Big) &\text{ if } \alpha-2L\leq \xi\leq \alpha-L,\\ 1 & \text{ if } \alpha-L\leq \xi\leq \alpha+L,\\ 1-h\Big(\frac{\xi-\alpha}{L}-1\Big) & \text{ if } \alpha+L\leq \xi\leq \alpha+2L, \\ 0 & \text{ if } \alpha+2L\leq \xi.\end{cases}\]
For \(l\geq 1\) we define 
\[b_l^{L,\alpha}(\xi)=\begin{cases} 0 & \text{ if } \xi\leq \alpha +2^{l-1}L,\\ h\Big(\frac{\xi-\alpha}{2^{l-1} L}-1\Big) & \text{ if } \alpha+2^{l-1}L\leq \xi\leq \alpha+2^l L,\\ 1-h\Big(\frac{\xi-\alpha}{2^{l} L}-1\Big) & \text{ if } \alpha+2^l L \leq \xi\leq \alpha +2^{l+1}L, \\ 0 & \text{ if } \alpha+2^{l+1}L\leq \xi.\end{cases}\]
Finally, for \(l\leq -1\) we define
\[b_l^{L,\alpha}(\xi)=b_{-l}^{L,\alpha}(2\alpha-\xi)=\begin{cases} 0 &\text{ if } \xi \leq \alpha -2^{|l|+1}L,\\
1-h\Big(\frac{\alpha-\xi}{2^{|l|}L}-1\Big) &\text{ if } \alpha -2^{|l|+1}L\leq \xi\leq \alpha-2^{|l|}L,\\ h\Big(\frac{\alpha-\xi}{2^{|l|-1}L}-1\Big) &\text{ if } \alpha-2^{|l|}L\leq \xi\leq \alpha-2^{|l|-1}L,\\ 0&\text{ if } \alpha -2^{|l|-1}L\leq\xi .\end{cases}\]
It is easy to check that the properties \ref{unityi} and \ref{unityiii} hold for this choice of \((b_l^{L,\alpha})_{l\in\mathbb{Z}}\).
\end{proof}

\subsection{Equidistribution}

This section is based on \cite[Ch.1]{montgomery10}. 
Suppose we have a real sequence \((u(n))_{n\geq1}\).
We say that \((u(n))\) \emph{equidistributes modulo 1} if for all \(0\leq \alpha\leq\beta\leq 1\) we have
\begin{equation*}
\lim_{N\to\infty} \frac{1}{N} \#\big\{ 1\leq n\leq N: u(n)\in[\alpha,\beta] \Mod 1\big\}=\beta-\alpha.
\end{equation*}
(By \(u(n)\in[\alpha,\beta] \Mod 1\), we mean that the fractional part \(\{u(n)\}=u(n)-\lfloor u(n)\rfloor\in[\alpha,\beta]\).)

To measure the quantitative rate of equidistribution modulo 1, we introduce the \emph{discrepancy} of the sequence \((u(n))\). 
For \(N\geq1\), the discrepancy \(D(u, N)\) is defined by 
\begin{equation}\label{discrepancydef}
D(u, N)\coloneqq \sup_{0\leq \alpha\leq\beta\leq 1}\Big|\#\big\{1\leq n\leq N: u(n)\in[\alpha,\beta]\Mod 1\big\}-N(\beta-\alpha)\Big|.
\end{equation}
Thus \((u(n))\) equidistributes modulo 1 if and only if \(D(u, N)=o(N)\) as \(N\to\infty\). 
We give the following bound for the discrepancy (which is essentially a quantitative version of Weyl's equidistribution theorem).

\begin{lemma}[Erd\H os-Tur\'an Inequality {\cite[Corollary 1.1]{montgomery10}}]\label{erdosturanlemma}
For any real sequence \((u(n))_{n\geq1}\), any \(N\geq1\) and any positive integer \(R\), we have
\begin{equation*}%\label{erdosturan}
D(u, N)\leq \frac{N}{R+1} +3\sum_{r\leq R}\frac{1}{r}\Big|\sum_{n\leq N} e(r u(n))\Big|.
\end{equation*}
\end{lemma}

\subsection{Oscillatory Sums and Integrals}

The following lemma is a formulation of van der Corput's method for bounding exponential sums. 

\begin{lemma}[van der Corput {\cite[Theorem 8.20]{iwanieckowalski}}]\label{vdcorput}
Let \(f\) be a real-valued function which is \(p\) times continuously differentiable, with \(p\geq 2\). Let \(b-a\geq1\), and suppose that if \(\xi\in[a,b]\), \(\lambda\leq f^{(p)}(\xi)\leq \mu\lambda\) for some \(\lambda>0\) and \(\mu\geq1\). Then
\[\sum_{a<n\leq b} e(f(n))\ll (b-a) \mu^{2/P} \lambda^{1/(2P-2)} +(b-a)^{1-2/P}\lambda^{-1/(2P-2)},\]
where \(P=2^{p-1}\) and the implied constants are independent of \(p\).
\end{lemma}

Finally, we record two lemmas that will be used later to handle oscillatory integrals. These can both be found in work of Blomer, Khan and Young \cite[§8]{stationaryphase}.

\begin{lemma}[Integration by Parts {\cite[Lemma 8.1]{stationaryphase}}]\label{bkyipb}
Let \(Q,U,R,X>0\) and \(Y\geq 1\) be some parameters. Let \(\rho\) and \(\phi\) be two smooth functions. Assume \(\rho\) is compactly supported on the interval \([\alpha,\beta]\) and satisfies
\[\rho^{(j)}(t)\ll_j XU^{-j} \text{ for } j=0,1,2\ldots\]
Assume \(\phi\) satisfies
\[\phi'(t)\gg R \text{ and } \phi^{(j)}(t)\ll_j YQ^{-j} \text{ for } j=2,3,\ldots\]
Then
\begin{equation}\label{intbound}
\int_\alpha^\beta \rho(t)e^{i\phi(t)}dt\ll_B (\beta-\alpha)X\Big\{\Big(\frac{QR}{Y^{1/2}}\Big)^{-B}+(RU)^{-B}\Big\},
\end{equation}
for any integer \(B\geq 0\).
\end{lemma}

The following is a simplified version of \cite[Proposition 8.2]{stationaryphase}, which gives an asymptotic expansion of oscillatory integrals around a stationary phase. 
We state only an upper bound here.

\begin{lemma}[Stationary Phase {\cite[Proposition 8.2]{stationaryphase}}]\label{bkyexpansion}
Let \(Q,V,X,Y>0\) be some parameters. 
Let \(\rho\) and \(\phi\) be two smooth functions, with \(\rho\) compactly supported on an interval \([\alpha,\beta]\) where \(\beta-\alpha \geq V\). Set \(Z\coloneqq Q+X+Y+(\beta-\alpha)+1\), and suppose that for some fixed \(\epsilon>0\) we have
\begin{equation}\label{bkyexpansioncondition}
Y\geq Z^\epsilon \text{ and } \frac{VY^{1/2}}{Q}\geq Z^\epsilon.
\end{equation}
Assume 
\[\rho^{(j)}(t)\ll_j XV^{-j} \text{ for } j=0,1,2,\ldots\]
Additionally, suppose there exists a unique point \(t_0\in[\alpha,\beta]\) satisfying \(\phi'(t_0)=0\). Assume further that
\[\phi''(t)\asymp YQ^{-2} \text{ and } \phi^{(j)}(t)\ll_j YQ^{-j} \text{ for } j=1,2,3,\ldots\]
Then for any integer \(B\geq0\), 
\begin{equation}\label{intbound2}
\int_\alpha^\beta \rho(t) e^{i\phi(t)}dt \ll_B \frac{QX}{Y^{1/2}}+Z^{-B}.
\end{equation}
\end{lemma}

\section{Bounds for Sums of Hecke Eigenvalues: Proof of Theorem \ref{thmbound}}\label{secbound}

To prove Theorem \ref{thmbound}, we will apply Lemma \ref{vorprop} to transform the sums (shortening their effective length). We first give the following bound for the weights appearing in the transformed sums, from which the theorem will follow straightforwardly.

\begin{lemma}\label{wgood}
For \(\Delta\geq1\), let \(w=w_\Delta\) be given as in Definition \ref{wdef} and \(\tilde w=\tilde w_\Delta\) as in (\ref{tildew}). Suppose \(x\geq k^2/(8\pi^2)\). Then for \(n\geq1\),
\begin{multline}\label{w1}
\tilde w\Big(\frac{nx}{k^2+\Delta^2}\Big)\\
=\frac{2\sqrt{2}}{\sqrt{\pi}}\int_0^\infty \Big\{\frac{12\pi^2nx}{(16\pi^2nxt-\kappa^2)}tw(t)-w(t)-tw'(t)\Big\}(16\pi^2nxt-\kappa^2)^{-3/4}\sin \omega(4\pi \sqrt{nxt})dt\\
+\mathcal O\big((nx)^{-5/4}\big).
\end{multline}
In particular, we have the bound
\begin{equation}\label{w2}
\tilde w\Big(\frac{nx}{k^2+\Delta^2}\Big)\ll (nx)^{-3/4}.
\end{equation}
\end{lemma}

\begin{proof} 
Integrating the expression given in Lemma \ref{wxbig} by parts (using (\ref{omega'again})), we have
\begin{multline*}
\tilde w\Big(\frac{nx}{k^2+\Delta^2}\Big)=\sqrt{\frac2\pi}\int_0^\infty w(t)(16\pi^2 nxt-\kappa^2)^{-1/4}\cos \omega(4\pi \sqrt{nxt})dt+\mathcal O((nx)^{-5/4})\\
=\frac{2\sqrt2}{\sqrt\pi}\int_0^\infty tw(t)(16\pi^2nxt-\kappa^2)^{-3/4}d\big(\sin \omega(4\pi \sqrt{nxt})\big)+\mathcal O((nx)^{-5/4})\\
=-\frac{2\sqrt2}{\sqrt\pi}\int_0^\infty \Big\{w(t)+tw'(t)-tw(t)\frac{3}{4}\frac{16\pi^2nx}{(16\pi^2nxt-\kappa^2)}\Big\}(16\pi^2nxt-\kappa^2)^{-3/4}\sin \omega(4\pi \sqrt{nxt})dt\\
+\mathcal O((nx)^{-5/4}),
\end{multline*}
proving (\ref{w1}). To deduce (\ref{w2}) from (\ref{w1}), one uses the bounds \(w\ll 1\), \(w'\ll \Delta\) and notes that \(w'\) is supported on a set of measure \(\leq 2/\Delta\).
\end{proof}

Theorem \ref{thmbound} is now a simple consequence of Lemma \ref{vorprop} and the bound (\ref{w2}) of Lemma \ref{wgood}.

\begin{proof}[Proof of Theorem \ref{thmbound}]
Set \(\Delta=x^{2/3}\) throughout this proof. We apply Lemma \ref{vorprop} with this choice of \(\Delta\), which shows
\begin{multline}\label{propproofsplit}
\mathcal S(x,f)=2\pi (-1)^{k/2} x\sum_{n\geq1} \lambda_f(n)\tilde w\Big(\frac{nx}{k^2+\Delta^2}\Big)+\mathcal O \Big(\frac{x\log x}{\Delta}\Big)\\
\ll x\sum_{n\geq 1} d(n)\Big|\tilde w\Big(\frac{nx}{k^2+\Delta^2}\Big)\Big|+x^{1/3}\log x.
\end{multline}
In the last line, we used Deligne's bound \(|\lambda_f(n)|\leq d(n)\). 

If \(n\geq x^{1/3+\epsilon}\), then since \(x^{1/3+\epsilon}= \Delta^2/x^{1-\epsilon}\geq (k^2+\Delta^2)/(2x^{1-\epsilon})\) (our assumption \(x\geq k^2/(8\pi^2)\) implies \(\Delta^2\geq k^2\)), (\ref{tildewibped}) of Lemma \ref{vorprop} gives the bound \(\tilde w(nx/(k^2+\Delta^2))\ll n^{-2}x^{-1000}\), say. If \(n\leq x^{1/3+\epsilon}\), we use the bound \(\tilde w(nx/(k^2+\Delta^2))\ll (nx)^{-3/4}\) of Lemma \ref{wgood}. From (\ref{propproofsplit}), we conclude the bound
\begin{equation*}
\mathcal S(x,f)\ll x^{1/4}\sum_{n\leq x^{1/3+\epsilon}}\frac{d(n)}{n^{3/4}}+x^{-999}\sum_{n\geq x^{1/3+\epsilon}}\frac{d(n)}{n^2}+x^{1/3}\log x\ll x^{1/3+\epsilon}.
\end{equation*}
\end{proof}

\section{The First Moment: Proof of Theorem \ref{thmmean}}\label{secmean}

To compute the first moment \(\langle \mathcal S(x,f)\rangle\), our strategy is to transform the sums using Lemma \ref{vorprop} and then apply the Petersson trace formula. We first need a more explicit form of the smoothings \(\tilde w\) appearing in the transformed sums, given in the following lemma.

\begin{lemma}\label{tildewibp}
For \(\Delta\geq1\), let \(w=w_\Delta\) be given as in Definition \ref{wdef} and \(\tilde w=\tilde w_\Delta\) as in (\ref{tildew}). Suppose \(x\geq k^2/(8\pi^2)\). Then for \(n\geq1\), 
\[\tilde w_\Delta \Big(\frac{nx}{k^2+\Delta^2}\Big)=\frac{2\sqrt{2}}{\sqrt{\pi}}\Omega(n,x)(nx)^{-3/4}+\mathcal O((nx)^{-5/4})+\mathcal O((nx)^{-1/4}\Delta^{-1}),\]
where \(\Omega(n,x)\) is given by (\ref{bigomegadef}). 
\end{lemma}

\begin{proof}
We use the expression given in Lemma \ref{wxbig}, and first unsmooth the integral there. From Definition \ref{wdef} and Lemma \ref{wxbig} we obtain
\begin{multline}\label{weights00}
\tilde w\Big(\frac{nx}{k^2+\Delta^2}\Big)=\sqrt{\frac2\pi} \int_1^2(16\pi^2nxt-\kappa^2)^{-1/4}\cos\omega(4\pi\sqrt{nxt})dt\\
+\mathcal O\Big(\Big\{\int_1^{1+\Delta^{-1}}+\int_{2-\Delta^{-1}}^2\Big\}(16\pi^2 nxt-\kappa^2)^{-1/4}dt\Big)+\mathcal O((nx)^{-5/4}).
\end{multline}
Since we assume \(x\geq k^2/(8\pi^2)\) (which implies \((16\pi^2nxt-\kappa^2)\asymp nx\)), the first error term here is \(\mathcal O((nx)^{-1/4}\Delta^{-1})\). Integrating the main term by parts (recall (\ref{omega'again})) we obtain
\begin{multline*}
\int_1^2(16\pi^2nxt-\kappa^2)^{-1/4}\cos\omega(4\pi\sqrt{nxt})dt=2\int_1^2 t(16\pi^2nxt-\kappa^2)^{-3/4}d(\sin\omega(4\pi\sqrt{nxt}))\\
=2\Omega(n,x)(nx)^{-3/4}-2\int_1^2\Big(1-\frac34\frac{16\pi^2nxt}{16\pi^2nxt-\kappa^2}\Big)(16\pi^2nxt-\kappa^2)^{-3/4}\sin\omega(4\pi\sqrt{nxt})dt.
\end{multline*}
Integrating by parts once again, the remaining integral is easily shown to be \(\mathcal O((nx)^{-5/4})\). Replacing this expression in (\ref{weights00}), the lemma is proved. 
\end{proof}

Equipped with this expression for the weights, Theorem \ref{thmmean} now follows straightforwardly from the Vorono\"i formula (Lemma \ref{vorprop}) and the Petersson trace formula (Lemma \ref{trace}). 

\begin{proof}[Proof of Theorem \ref{thmmean}]
Throughout this proof, we fix \(\epsilon>0\) and set \(\Delta=x^{1/2}k^{1-\epsilon}\). 
We first wish to apply Lemma \ref{vorprop} with this choice of \(\Delta\) (applicable since \(k^4\geq x\geq k^2/(8\pi^2)\implies \Delta \leq x^{1-\epsilon/5}\)). 
Since for \(n\geq (k^2+\Delta^2)k^{\epsilon}/x\), (\ref{tildewibped}) provides the bound \(\tilde w(nx/(k^2+\Delta^2))\ll n^{-2}k^{-1100}\), Lemma \ref{vorprop} (with \(\Delta=x^{1/2}k^{1-\epsilon }\)) shows that for \(f\in\mathcal B_k\), 
\begin{multline*}
\mathcal S(x,f)=2\pi(-1)^{k/2}x\sum_{n\geq1}\lambda_f(n)\tilde w\Big(\frac{nx}{k^2+\Delta^2}\Big)+\mathcal O \Big(\frac{x\log x}{\Delta}\Big)\\
=2\pi(-1)^{k/2}x\sum_{n\leq (k^2+\Delta^2)k^{\epsilon }/x}\lambda_f(n)\tilde w\Big(\frac{nx}{k^2+\Delta^2}\Big)+\mathcal O(k^{-1000})+\mathcal O(x^{1/2}k^{-1+\epsilon }\log x).
\end{multline*}
Since \((k^2+\Delta^2)k^{\epsilon }/x\leq k^2/10^4\), we may apply part \ref{traceii} of Lemma \ref{trace} to compute
\begin{multline*}
\langle \mathcal S(x,f)\rangle =2\pi(-1)^{k/2} x\sum_{n\leq (k^2+\Delta^2)k^{\epsilon }/x}\langle \lambda_f(n)\lambda_f(1)\rangle \tilde w\Big(\frac{nx}{k^2+\Delta^2}\Big)+\mathcal O(x^{1/2}k^{-1+2\epsilon })\\
=2\pi(-1)^{k/2} x\tilde w\Big(\frac{x}{k^2+\Delta^2}\Big)+\mathcal O\Big(x\sum_{n\leq (k^2+\Delta^2)k^{\epsilon }/x}\Big|\tilde w\Big(\frac{nx}{k^2+\Delta^2}\Big)\Big|\cdot e^{-k}\Big)+\mathcal O(x^{1/2}k^{-1+2\epsilon}).
\end{multline*}
It is clear (from the bound (\ref{w2}), say) that the first error term is \(\mathcal O(e^{-k/2})\). Now applying Lemma \ref{tildewibp}, we conclude 
\begin{equation*}
\langle \mathcal S(x,f)\rangle =4\sqrt{2\pi}(-1)^{k/2}\Omega(1,x)x^{1/4}+\mathcal O(x^{-1/4})+\mathcal O(x^{1/4}k^{-1+\epsilon })+\mathcal O(x^{1/2}k^{-1+2\epsilon}).
\end{equation*}
This completes the proof (upon replacing \(\epsilon\) by \(\epsilon/2\)).
\end{proof}

\section{The Second Moment: Proof of Theorem \ref{thmvar}}\label{secvar}

In this section, we compute the second moment \(\langle \mathcal S(x,f)^2\rangle\) in the range \(k^2/(8\pi^2)\leq x\leq k^{12/5}\). We begin with the following lemma, which applies the Petersson trace formula to naturally split the second moment into diagonal and off-diagonal terms. 

\begin{lemma}\label{varsplit}
Let \(\epsilon>0\). Let \(\Delta\) be a sufficiently large parameter satisfying \(x^{1/2}\leq \Delta\leq x^{1-\epsilon}\). Let \(w=w_\Delta\) and \(\tilde w=\tilde w_\Delta\) be the corresponding smoothings, given in Definition \ref{wdef} and (\ref{tildew}) respectively. Suppose \(k^2/(8\pi^2)\leq x\leq k^4\). Then we have
\begin{equation*}
\langle \mathcal S(x,f)^2\rangle = \sigma^2+\mathcal O \Big(\frac{\sigma x\log x}{\Delta}\Big)+\mathcal O \Big(\frac{x^2\log^2 x}{\Delta^2}\Big), 
\end{equation*}
where
\begin{equation}\label{splits}
\sigma^2=(\mathrm{D})+(\mathrm{OD}).
\end{equation}
Here \((\mathrm{D})\) denotes the diagonal terms
\begin{equation}\label{diag}
(\mathrm{D})=4\pi^2 x^2 \sum_{n\leq \Delta^2 k^\epsilon/x} \tilde w\Big(\frac{nx}{k^2+\Delta^2}\Big)^2.
\end{equation}
The off-diagonal terms are given by 
\begin{equation}\label{offdiag}
(\mathrm{OD})=8\pi^3(-1)^{k/2} x^2 \sum_{m,n\leq \Delta^2k^\epsilon/x} \tilde w\Big(\frac{nx}{k^2+\Delta^2}\Big)\tilde w\Big(\frac{mx}{k^2+\Delta^2}\Big) \sum_{c\geq1} c^{-1}S(m,n;c)J_{k-1}\Big(\frac{4\pi \sqrt{mn}}{c}\Big).
\end{equation}
\end{lemma}

\begin{proof}
Recall Lemma \ref{vorprop}, which shows that for \(f\in\mathcal B_k\), 
\begin{equation}\label{vorpropped}
\mathcal S(x,f)=2\pi (-1)^{k/2}x\sum_{n\geq1} \lambda_f(n)\tilde w\Big(\frac{nx}{k^2+\Delta^2}\Big)+\mathcal O \Big(\frac{x\log x}{\Delta}\Big).
\end{equation}
We claim we may truncate the sum over \(n\) in (\ref{vorpropped}) at \(\Delta^2 k^\epsilon/x\). Indeed, since we assume \(\Delta^2 \geq x\geq k^2/(8\pi^2)\), we have 
%\[n\geq \frac{\Delta^2 k^\epsilon}{x} \geq \frac{\Delta^2 k^\epsilon}{2x}+\frac{k^{2+\epsilon}}{16\pi^2 x}\geq \frac{(k^2+\Delta^2)k^\epsilon}{16\pi^2 x}\implies \tilde w\Big(\frac{nx}{k^2+\Delta^2}\Big)\ll n^{-2}k^{-1100},\]
\[n\geq \frac{\Delta^2 k^\epsilon}{x}\implies n\geq \frac{(k^2+\Delta^2)k^\epsilon}{16\pi^2 x}\implies \tilde w\Big(\frac{nx}{k^2+\Delta^2}\Big)\ll n^{-2}k^{-1100},\]
say, by the bound (\ref{tildewibped}) of Lemma \ref{vorprop}.
It follows from (\ref{vorpropped}) that 
\begin{equation*}
\mathcal S(x,f)=2\pi (-1)^{k/2}x\sum_{n\leq \Delta^2 k^\epsilon/x} \lambda_f(n)\tilde w\Big(\frac{nx}{k^2+\Delta^2}\Big)+\mathcal O \Big(\frac{x\log x}{\Delta}\Big).
\end{equation*}

This allows us to compute
\begin{multline}\label{var1}
\langle \mathcal S(x,f)^2\rangle = 4\pi^2 x^2 \sum_{m,n\leq \Delta^2k^\epsilon/x} \langle \lambda_f(n)\lambda_f(m)\rangle \tilde w\Big(\frac{nx}{k^2+\Delta^2}\Big)\tilde w\Big(\frac{mx}{k^2+\Delta^2}\Big)\\
+\mathcal O\Big(\Big\langle\Big|x\sum_{n\leq \Delta^2k^\epsilon/x} \lambda_f(n) \tilde w\Big(\frac{nx}{k^2+\Delta^2}\Big)\Big|\cdot \frac{x\log x}{\Delta}\Big\rangle\Big)+\mathcal O \Big(\frac{x^2\log^2 x}{\Delta^2}\Big).
\end{multline}
Set
\begin{equation}\label{vardef}
\sigma^2=4\pi^2 x^2 \sum_{m,n\leq \Delta^2 k^\epsilon/x} \langle \lambda_f(n)\lambda_f(m)\rangle \tilde w\Big(\frac{nx}{k^2+\Delta^2}\Big)\tilde w\Big(\frac{mx}{k^2+\Delta^2}\Big).
\end{equation}
Then
\begin{equation*}
\langle \mathcal S(x,f)^2\rangle = \sigma^2+\mathcal O \Big(\frac{\sigma x\log x}{\Delta}\Big)+\mathcal O \Big(\frac{x^2\log^2 x}{\Delta^2}\Big), 
\end{equation*}
as claimed (we have obtained the first error term here by applying the Cauchy-Schwarz inequality to the first error term of (\ref{var1})). The explicit expression (\ref{splits}) for \(\sigma^2\) follows immediately, upon applying the Petersson trace formula (part \ref{tracei} of Lemma \ref{trace}) to (\ref{vardef}).
\end{proof}

The diagonal terms \((\mathrm{D})\) given in (\ref{diag}) will contribute the main term in Theorem \ref{thmvar}; the off-diagonal terms \((\mathrm{OD})\) given in (\ref{offdiag}) will contribute only an error term. In Section \ref{subsectdiag}, we will evaluate the diagonal contribution, and prove the following.

\begin{restatable}[Evaluation of \((\mathrm D)\)]{lemma}{diaglem}\label{diaglem}
Let \(\epsilon>0\), and assume \(k^2/(8\pi^2)\leq x\leq k^4\) and \(x^{1/2}\leq \Delta\leq x^{1-\epsilon}\). Then the following hold.
\begin{enumerate}[label=(\roman*)]
\item \label{diagi} 
We have 
\begin{equation*}
(\mathrm{D})
=
32\pi x^{1/2}\sum_{n\geq1}\frac{\Omega(n,x)^2}{n^{3/2}}
+
\mathcal O(xk^\epsilon\Delta^{-1}).
\end{equation*}
\item \label{diagii} 
Moreover, the main term above satisfies
\begin{equation*}
x^{1/2}\exp\Big(-\frac{\log x}{\log\log x}\Big) 
\ll
32\pi x^{1/2}\sum_{n\geq1}\frac{\Omega(n,x)^2}{n^{3/2}} 
\ll 
x^{1/2}.
\end{equation*}
\end{enumerate}
\end{restatable}

Bounding the off-diagonal contribution (\ref{offdiag}) is the most difficult part of the proof of Theorem \ref{thmvar}. In Section \ref{subsecoffdiag}, we obtain the following.

\begin{restatable}[Bound for \((\mathrm{OD})\)]{lemma}{offdiaglem}\label{offdiaglem}
Let \(0<\epsilon<1/1000\), and assume \(x^{1/2}\leq \Delta\leq x^{2/3}k^{1/3-\epsilon}\). Provided \(k^2/(8\pi^2)\leq x\leq k^4\), we have
\begin{equation*}
(\mathrm{OD})\ll k^{-8/3}\Delta^2+k^{-3/2+\epsilon}\Delta^{3/2}+k^{-1/3+\epsilon}\Delta^{1/2}+x^{-1/2} k^{1/6+2\epsilon} \Delta+x^{-3/2}k^{-5/6+4\epsilon}\Delta^3.
\end{equation*}
\end{restatable}

\begin{proof}[Proof of Theorem \ref{thmvar}, assuming Lemma \ref{diaglem} and Lemma \ref{offdiaglem}]
We will apply Lemma \ref{offdiaglem} with \(\Delta=x^{1/2}k^{3/5}\) to bound the off-diagonal contribution. 
One easily checks that our assumption \(x\geq k^2/(8\pi^2)\) (and the fact \(\epsilon<1/1000\)) implies that this choice of \(\Delta\) satisfies \(\Delta=x^{1/2}k^{3/5}\leq x^{2/3}k^{1/3-\epsilon}\).
Lemma \ref{offdiaglem} therefore shows
\begin{equation*}
(\mathrm{OD})\ll x k^{-22/15}+ x^{3/4}k^{-3/5+\epsilon}+x^{1/4}k^{-1/30+\epsilon}+k^{23/30+2\epsilon}+k^{29/30+4\epsilon}.
\end{equation*}
Our assumption \(k^2/(8\pi^2)\leq x\leq k^{12/5}\) implies that the first and third term above are dominated by the second term. 
Hence
%Our assumption \(x\leq k^{12/5}\) implies \(xk^{-22/15}\leq x^{3/4}k^{-13/15}\leq x^{3/4}k^{-3/5+\epsilon}\), and the assumption \(x\geq k^2/(8\pi^2)\) implies \(x^{1/4}k^{-1/30+\epsilon}\leq 2\pi\sqrt2 x^{3/4}k^{-31/30+\epsilon}\leq x^{3/4}k^{-3/5+\epsilon}\), so the above bound may be simplified to  
\begin{equation*}
(\mathrm{OD})\ll x^{3/4}k^{-3/5+\epsilon}+k^{29/30+4\epsilon}.
\end{equation*}

Observe from (\ref{splits}) and Lemma \ref{diaglem} (applied with our choice of \(\Delta=x^{1/2}k^{3/5}\)) that
\begin{equation}\label{sigmarevisited}
\sigma^2=(\mathrm{D})+(\mathrm{OD})=32\pi x^{1/2}\sum_{n\geq1}\frac{\Omega(n,x)^2}{n^{3/2}}+\mathcal O(x^{3/4}k^{-3/5+\epsilon})+\mathcal O(k^{29/30+4\epsilon}).
\end{equation}
Using the bound for the above main term given in part \ref{diagii} of Lemma \ref{diaglem}, and the assumption \(x\leq k^{12/5}\), we find \(\sigma\ll x^{1/4}k^{\epsilon/2}\).
So from (\ref{sigmarevisited}) and Lemma \ref{varsplit}, we have for \(k^2/(8\pi^2)\leq x\leq k^{12/5}\)
\begin{multline*} 
\langle \mathcal S(x,f)^2\rangle=\sigma^2+\mathcal O\Big(\frac{\sigma x\log x}{\Delta}\Big)+\mathcal O\Big(\frac{x^2\log^2x}{\Delta^2}\Big)\\
=32\pi x^{1/2}\sum_{n\geq1}\frac{\Omega(n,x)^2}{n^{3/2}}+\mathcal O(x^{3/4}k^{-3/5+\epsilon})+\mathcal O(k^{29/30+4\epsilon}).
\end{multline*}
(We used that \(\mathcal O(x^2\log ^2x/\Delta^2)=\mathcal O(x^{3/4}k^{-3/5+\epsilon})\), since \(x\leq k^{12/5}\).) This completes the proof of Theorem \ref{thmvar} (upon replacing \(\epsilon\) by \(\epsilon/4\)).
\end{proof}

\subsection{Sketch of Argument}

\subsubsection{The Diagonal Terms: Sketch of Lemma \ref{diaglem}}
Part \ref{diagi} of Lemma \ref{diaglem} states
(for \(\Omega(n,x)\) given by (\ref{bigomegadef})) 
\begin{equation*}
(\mathrm{D})
=
32\pi x^{1/2}\sum_{n\geq1}\frac{\Omega(n,x)^2}{n^{3/2}}
+
\mathcal O(xk^\epsilon\Delta^{-1}).
\end{equation*}
We obtain this as a simple consequence of Lemma \ref{tildewibp}, which gives an explicit expression for the smoothings \(\tilde w\) appearing in (\ref{diag}). 
It is reasonable to guess that main term above should be of size \(\asymp x^{1/2}\). 
Indeed, for \(x\geq k^2/(8\pi^2)\) we have the bound \(\Omega(n,x)\ll 1\), and so it is immediate that in this range
\begin{equation}\label{maintermupperbound} 
32\pi x^{1/2}\sum_{n\geq1}\frac{\Omega(n,x)^2}{n^{3/2}} 
%\ll x^{1/2} \sum_{n\geq1} \frac{1}{n^{3/2}}
\ll x^{1/2}.
\end{equation}
%The more difficult part of the proof of Lemma \ref{diaglem} is establishing the lower bound for the main term (given in part \ref{diagii}). 
%The main term behaves roughly like
%\[x^{1/2}\sum_{n\geq 1}\frac{\sin^2\varphi(n)}{n^{3/2}},\] 
%for some complicated phase function \(\varphi\) depending on \(x\). Since all terms in this sum are positive, for the lower bound it suffices to show that the phase \(\varphi (n)\) is equidistributed modulo \(2\pi\) (since then not all the oscillatory terms can be small simultaneously). This is achieved by a standard application of the Erd\H os-Tur\'an inequality.
%Proving part \ref{diagii} is much more difficult, and we devote the remainder of the section to this task. It is immediate that \(x^{1/2}\sum_{n}\Omega(n,x)^2/n^{3/2}\ll x^{1/2}\), but it is more difficult to establish an inequality in the opposite direction. 
To prove the corresponding lower bound \(x^{1/2}\sum_{n}\Omega(n,x)^2/n^{3/2}\gg x^{1/2}\), it would of course suffice to show that for at least one \(n\leq 1000\), say, we have \(|\Omega(n,x)|\geq 1/1000\). Unfortunately, due to the presence of the oscillatory terms \(\sin\omega(4\pi\sqrt{nx})\) and \(\sin\omega(4\pi\sqrt{2nx})\) in the definition (\ref{bigomegadef}) of \(\Omega(n,x)\), it is possible that no such condition holds for some perverse choice of \(x\).

The most difficult part of the proof of Lemma \ref{diaglem} is establishing the lower bound for the main term, \(x^{1/2}\sum_{n}\Omega(n,x)^2/n^{3/2}\gg x^{1/2}\exp(-\log x/\log\log x)\), given in part \ref{diagii}. It is certainly reasonable to expect that for at least some \(n\), we have \(\Omega(n,x)\gg1\). 
%\begin{equation}\label{lowerboundideal}
%\Omega(n,x)\gg 1.
%\end{equation}
Indeed, define the sequence \((h(n))_{n\geq1}\) by 
\begin{equation}\label{hdef}
h(n)\coloneqq \frac{\omega(4\pi \sqrt{2nx})}{2\pi},
\end{equation}
and for \(N\geq 4\) set 
\begin{equation}\label{zndef}
Z(N)\coloneqq \left\{4\leq n\leq N : h(n)\in \left[\frac{9}{40},\frac{11}{40}\right] \Mod 1\right\}.
\end{equation}
Then it is not hard to see that (provided \(x\geq k^2/(8\pi^2)\)),
\begin{equation}\label{znimplication}
n\in Z(N)\implies \Omega(n,x)\gg1.
\end{equation}
%Indeed, sufficient conditions for (\ref{lowerboundideal}) to hold are \(n\geq 4\) and
%\begin{equation}\label{equidistcond}
%h(n)\coloneqq \frac{\omega(4\pi\sqrt{2nx})}{2\pi}\in \left[\frac{9}{40},\frac{11}{40}\right] \Mod 1.
%\end{equation}
%(In other words the fractional part \(\{h(n)\}\coloneqq h(n)-\lfloor h(n)\rfloor\in[9/40,11/40]\).) 

To check (\ref{znimplication}), observe
\begin{equation*}
h(n)\in \left[\frac{9}{40},\frac{11}{40}\right]\Mod 1 
\implies
 \sin\omega(4\pi\sqrt{2nx})\geq \sin\Big(\frac{9\pi}{20}\Big)\geq c,
\end{equation*}
%\begin{multline*} 
%h(n)\in \left[\frac{9}{40},\frac{11}{40}\right]\Mod 1 \implies \omega(4\pi\sqrt{2nx})\in\Big[\frac{9\pi}{20}, \frac{11\pi}{20}\Big] \Mod{2\pi}\\
%\implies \sin\omega(4\pi\sqrt{2nx})\geq \sin\Big(\frac{9\pi}{20}\Big)\geq c,
%\end{multline*}
where (for convenience later on) we define the constant 
\[c\coloneqq 2^{1/2}\cdot (7/4)^{-3/4}\cdot \frac{101}{100}= 0.938\ldots\]
(\(\sin(9\pi/20)=0.987\ldots\) for comparison). Consequently, if \(n\in Z(N)\) then
\begin{multline*}
|\Omega(n,x)|
=
|2(32\pi^2 -\kappa^2/(nx))^{-3/4}\sin\omega(4\pi\sqrt{2nx})-(16\pi^2 -\kappa^2/(nx))^{-3/4}\sin\omega(4\pi\sqrt{nx})|\\
\geq 
2c(32\pi^2 -\kappa^2/(nx))^{-3/4}-(16\pi^2 -\kappa^2/(nx))^{-3/4}.
\end{multline*}
%\begin{align*}
%|\Omega(n,x)|&=|2(32\pi^2 -\kappa^2/(nx))^{-3/4}\sin\omega(4\pi\sqrt{2nx})-(16\pi^2 -\kappa^2/(nx))^{-3/4}\sin\omega(4\pi\sqrt{nx})|\\
%&\geq 2(32\pi^2 -\kappa^2/(nx))^{-3/4}\sin\omega(4\pi\sqrt{2nx})-(16\pi^2 -\kappa^2/(nx))^{-3/4}\\
%&\geq 2c(32\pi^2 -\kappa^2/(nx))^{-3/4}-(16\pi^2 -\kappa^2/(nx))^{-3/4}.
%\end{align*}
Note that \(n\in Z(N)\implies n\geq4\implies  2-1/n\geq 7/4\). 
Using this, and the assumption \(x\geq k^2/(8\pi^2)\), it follows that \(n\in Z(N)\) implies
\begin{equation*}
|\Omega(n,x)|
\geq
2c(32\pi^2)^{-3/4}-(8\pi^2)^{-3/4}(7/4)^{-3/4}
\geq 
(8\pi^2)^{-3/4} (7/4)^{-3/4}(101/100-1)
\gg 1.
\end{equation*}
%\begin{align}\label{equidistcondproof}
%|\Omega(n,x)|&\geq 2c(32\pi^2)^{-3/4}-(16\pi^2 -8\pi^2/n)^{-3/4}\\
%&\geq 2c(32\pi^2)^{-3/4}-(8\pi^2)^{-3/4}(7/4)^{-3/4}\nonumber\\
%&=(8\pi^2)^{-3/4}(2^{-1/2} c-(7/4)^{-3/4})\nonumber\\
%&\geq (8\pi^2)^{-3/4} (7/4)^{-3/4}(101/100-1)
%\gg 1.\nonumber
%\end{align}
So (\ref{znimplication}) is proved.

Our strategy to prove the lower bound in part \ref{diagii} of Lemma \ref{diaglem} is to show that the sequence \((h(n))_{n\geq1}\) equidistributes modulo 1. 
That ensures that \(Z(N)\) contains `many' integers.
So (see (\ref{znimplication})), many \(n\) satisfy \(\Omega(n,x)\gg 1\).
It is these \(n\) which will give a large contribution to the main term \(x^{1/2}\sum_{n}\Omega(n,x)^2/n^{3/2}\), leading to the desired lower bound.

\subsubsection{The Off-diagonal Terms: Sketch of Lemma \ref{offdiaglem}}

The basic idea behind the proof is to reorder summation in the definition (\ref{offdiag}) and apply Poisson summation in the \(n\) variable. Very roughly speaking, after reordering summation we have
\begin{equation*}
(\mathrm{OD})\approx x^{2} \sum_{c\leq \Delta^2/(kx)}c^{-1}\sum_{m\leq \Delta^2/x}\tilde w\Big(\frac{mx}{k^2+\Delta^2}\Big) \sum_{n\leq \Delta^2/x} S(m,n;c)\tilde w\Big(\frac{nx}{k^2+\Delta^2}\Big)J_{k-1}\Big(\frac{4\pi\sqrt{mn}}{c}\Big).
\end{equation*}
(The effective range of \(c\) is a consequence of the effective support of the Bessel function \(J_{k-1}(4\pi\sqrt{mn}/c)\).) Opening the Kloosterman sum (writing \(a^*\) for the multiplicative inverse of \(a\) modulo \(c\)), we can rewrite the above as
\begin{align*}
&\approx x^{2} \sum_{c\leq \Delta^2/(kx)}c^{-1}\sum_{\substack{1\leq a\leq c\\(a,c)=1}}\sum_{m\leq \Delta^2/x} e\Big(\frac{a^*m}{c}\Big)\tilde w\Big(\frac{mx}{k^2+\Delta^2}\Big)\\
&\hspace{12em}\sum_{n\leq \Delta^2/x} e\Big(\frac{an}{c}\Big)\tilde w\Big(\frac{nx}{k^2+\Delta^2}\Big)J_{k-1}\Big(\frac{4\pi\sqrt{mn}}{c}\Big).
\end{align*}
We will use
\begin{equation}\label{terribleapproximations}
\tilde w\Big(\frac{mx}{k^2+\Delta^2}\Big)\ll (mx)^{-3/4} \text{ and } \tilde w\Big(\frac{nx}{k^2+\Delta^2}\Big)\approx (nx)^{-3/4} \sin\varphi(n),
\end{equation}
where \(\varphi\) is some complicated phase function depending on \(x\). The first bound is (\ref{w2}) of Lemma \ref{wgood}, and the latter is a very rough approximation of (\ref{w1}) (which suffices for this sketch). Applying the triangle inequality and (\ref{terribleapproximations}), it suffices to bound a sum of shape
\begin{equation*}
\approx x^{1/2}\sum_{c\leq \Delta^2/(kx)}c^{-1}\sum_{m\leq \Delta^2/x} m^{-3/4} \sum_{\substack{1\leq a\leq c\\(a,c)=1}} \Big|\sum_{n\leq \Delta^2/x} e\Big(\frac{an}{c}\Big)n^{-3/4} \sin\varphi(n) J_{k-1}\Big(\frac{4\pi\sqrt{mn}}{c}\Big)\Big|.
\end{equation*}

The difficult part of the proof is bounding the inner sum over \(n\). One would expect to find cancellation in this sum, coming from the oscillations of the exponential terms, and this is captured by applying Poisson summation. Indeed, this shows that the inner sum over \(n\) is roughly equal to the dual sum
\[\sum_{\substack{\tilde n\in\mathbb{Z}\\\tilde n\equiv -a \Mod c}} \int_0^{\Delta^2/x} t^{-3/4}\sin\varphi(t)J_{k-1}\Big(\frac{4\pi\sqrt{mt}}{c}\Big)e\Big(-\frac{\tilde nt}{c}\Big)dt.\]
Integrating by parts, one finds that this dual sum is effectively of length \(\approx \sqrt{mx}/k (\ll \Delta/k^{1-\epsilon})\). This is shorter than the original sum, and therefore we find some savings.

It remains to estimate the oscillatory integrals appearing in the dual sum, which requires some careful analysis of the Bessel function \(J_{k-1}(4\pi\sqrt{mt}/c)\). Handling the integrals using the method of stationary phase, we are able to show that the contribution of the off-diagonal terms is asymptotically smaller than that of the diagonal main term in the range \(x\leq k^{12/5-\epsilon}\).

Our proof is similar to work of Hough \cite{hough}, in which a twisted second moment estimate (also in weight aspect) is established for central values of \(L(s,f)\). After approximating the central values by Dirichlet polynomials (with coefficients \(\lambda_f(n)\)) and averaging with the Petersson trace formula, the off-diagonal contribution arising is handled in the same way.

\subsection{The Diagonal Terms: Proof of Lemma \ref{diaglem}}\label{subsectdiag}

In this section we investigate the diagonal contribution, which gives the main term in Theorem \ref{thmvar}. Recall the definition (\ref{diag}), which states
\begin{equation*}
(\mathrm{D})=4\pi^2 x^2 \sum_{n\leq \Delta^2 k^\epsilon/x} \tilde w\Big(\frac{nx}{k^2+\Delta^2}\Big)^2.
\end{equation*}
Our goal is the following. 
\diaglem*

\begin{proof}[Proof of Lemma \ref{diaglem}, part \ref{diagi}.]
The explicit expression for the diagonal terms is a simple consequence of Lemma \ref{tildewibp}, which states
\begin{equation*}
\tilde w\Big(\frac{nx}{k^2+\Delta^2}\Big)=\frac{2\sqrt{2}}{\sqrt{\pi}}\Omega(n,x)(nx)^{-3/4}+\mathcal O ((nx)^{-5/4})+\mathcal O((nx)^{-1/4}\Delta^{-1}).
\end{equation*}
From the definition (\ref{diag}) and the bound \(\Omega(n,x)\ll 1\) we compute
\begin{align*}
(\mathrm{D})&=4\pi^2x^2\sum_{n\leq \Delta^2k^\epsilon/x}\Big(\frac{2\sqrt{2}}{\sqrt\pi}\Omega(n,x)(nx)^{-3/4} +\mathcal O((nx)^{-5/4})+\mathcal O((nx)^{-1/4}\Delta^{-1})\Big)^2\\
&=32\pi x^{1/2}\sum_{n\leq \Delta^2k^\epsilon/x}\frac{\Omega(n,x)^2}{n^{3/2}}+\mathcal O\Big(x^2\sum_{n\leq \Delta^2k^\epsilon/x} (nx)^{-1}\Delta^{-1}\Big)\\
&\hspace{15em}+\mathcal O\Big(x^2\sum_{n\leq \Delta^2k^\epsilon/x}(nx)^{-1/2}\Delta^{-2}\Big)+\mathcal O(1)\\
&=32\pi x^{1/2}\sum_{n\geq1}\frac{\Omega(n,x)^2}{n^{3/2}}+\mathcal O\Big(x^{1/2}\sum_{n\geq \Delta^2k^\epsilon/x}n^{-3/2}\Big)+\mathcal O\Big(\frac{x\log\Delta}{\Delta}\Big)+\mathcal O \Big(\frac{x k^{\epsilon/2}}{\Delta}\Big)+\mathcal O(1)\\
&=32\pi x^{1/2}\sum_{n\geq1}\frac{\Omega(n,x)^2}{n^{3/2}}+\mathcal O \Big(\frac{x k^{\epsilon}}{\Delta}\Big),
\end{align*}
as claimed.
\end{proof}

The proof of part \ref{diagii} of Lemma \ref{diaglem} is more involved.
The key input is the following bound for the discrepancy of the sequence \((h(n))_{n\geq1}\). 
The sequence is given by \(h(n)=\omega(4\pi \sqrt{2nx})/(2\pi)\) (see (\ref{hdef})), and the discrepancy \(D(h,N)\) is defined for \(N\geq1\) by (\ref{discrepancydef}).

\begin{lemma}\label{hdiscrepancylem}
Let \(N\geq1\). Then for any integers \(R\geq1\) and \(p\geq2\), we have 
\begin{equation*}
D(h,N)
\ll
N\big( R^{-1}+2^p x^{\frac{1}{2(2P-2)}}N^{-\frac{2p-1}{2(2P-2)}}R^{\frac{1}{2P-2}} + x^{-\frac{1}{2(2P-2)}}N^{\frac{2p-1}{2(2P-2)}-\frac{2}{P}}\log R+pN^{-1}\log R\big), 
\end{equation*}
where \(P=2^{p-1}\).
\end{lemma}

We now prove the final part of Lemma \ref{diaglem}, then return to prove Lemma \ref{hdiscrepancylem} at the end of this section.

\begin{proof}[Proof of Lemma \ref{diaglem}, part \ref{diagii}, assuming Lemma \ref{hdiscrepancylem}]
The upper bound is trivial, see (\ref{maintermupperbound}).
For the lower bound, let \(N\geq4\) and recall from (\ref{zndef}) and (\ref{znimplication}) that
\begin{equation*}
n\in Z(N)
= \left\{4\leq n\leq N : h(n)\in \left[\frac{9}{40},\frac{11}{40}\right] \Mod 1\right\}
\implies \Omega(n,x)\gg1.
\end{equation*}
Consequently, we obtain the lower bound
\begin{equation}\label{genbounddiag}
32\pi x^{1/2}\sum_{n\geq 1} \frac{\Omega(n,x)^2}{n^{3/2}}
\geq 
32\pi x^{1/2}\sum_{n\in Z(N)} \frac{\Omega(n,x)^2}{n^{3/2}}
\gg
x^{1/2}N^{-3/2}\#Z(N).
\end{equation}
%To optimise the bound (\ref{genbounddiag}), we wish to choose \(N\) small enough that \(x^{1/2}N^{-3/2}\) is of comparable size to \(x^{1/2}\). However, we must also take \(N\) large enough to ensure that \(\#Z(N)\) is non-zero. To achieve the latter condition, we will choose \(N\) large enough that \(D(h,N)=o(N)\), whence (\ref{zasymp}) implies \(\#Z(N)\sim N/20\).
We therefore seek lower bounds for \(\# Z(N)\). 

We have the following estimate (see (\ref{discrepancydef})): 
\begin{equation*}
\#Z(N)=\frac{N}{20}+\mathcal O(1)+\mathcal O(D(h,N)).
\end{equation*}
We now apply Lemma \ref{hdiscrepancylem}. This shows that for any positive integers \(R\geq1\) and \(p\geq2\), 
\begin{multline}\label{zasymp}
\#Z(N)=\frac{N}{20}\big(1
+\mathcal O\big(R^{-1}\big)
+\mathcal O\big(2^p x^{\frac{1}{2(2P-2)}}N^{-\frac{2p-1}{2(2P-2)}}R^{\frac{1}{2P-2}}\big)\\
+\mathcal O\big(x^{-\frac{1}{2(2P-2)}}N^{\frac{2p-1}{2(2P-2)}-\frac{2}{P}}\log R\big)
+\mathcal O(pN^{-1}\log R\big)\big)+\mathcal O(1),
\end{multline}
where \(P=2^{p-1}\).
Choose 
\begin{equation}\label{choices} 
p=\Big\lfloor \frac12(\log\log x+3)\Big\rfloor,\; N=\Big\lfloor x^{\frac{1}{(2p-3)}}\Big\rfloor, \text{ and } R=\Big\lfloor x^{\frac{1}{(2p-3)(2P-1)}}\Big\rfloor.
\end{equation}
We now bound the error terms in (\ref{zasymp}) with these choices of parameters. Firstly, it is simple to note
\begin{equation*}
R\geq x^{\frac{1}{(2p-3)(2P-1)}} -1
%\geq \exp\Big(\frac{\log x}{\log\log x\cdot 2^{(\log\log x+3)/2}}\Big)-1=
\geq \exp\Big(\frac{(\log x)^{1-\log2/2}}{2^{3/2}\log\log x}\Big)-1\to\infty,
\end{equation*}
in other words \(R^{-1}=o(1)\). 
Next considering the second error term in (\ref{zasymp}), one has
%\begin{align*}
%2^p x&^{\frac{1}{2(2P-2)}}N^{-\frac{2p-1}{2(2P-2)}}R^{\frac{1}{2P-2}} \leq 2^{(\log\log x+3)/2}x^{\frac{1}{2(2P-2)}} (x^{\frac{1}{2p-3}}-1)^{-\frac{2p-1}{2(2P-2)}}x^{\frac{1}{(2p-3)(2P-2)(2P-1)}}\\
%&\ll (\log x)^{\log 2/2}\exp\Big(\log x\Big(\frac{1}{2(2P-2)}-\frac{2p-1}{2(2p-3)(2P-2)}+\frac{1}{(2p-3)(2P-2)(2P-1)}\Big)\Big)\\
%&=\exp\Big(-\frac{\log x}{(2p-3)(2P-1)}+\frac{\log2}{2}\log \log x\Big)\\
%&\leq \exp\Big(-\frac{(\log x)^{1-\log2/2}}{2^{3/2} \log\log x} +\frac{\log 2}{2}\log\log x\Big)=o(1).
%\end{align*}
\begin{multline*}
2^p x^{\frac{1}{2(2P-2)}}N^{-\frac{2p-1}{2(2P-2)}}R^{\frac{1}{2P-2}}
\ll \exp\Big(\frac{\log2}{2}\log\log x+\frac{\log x}{2P-2}\Big(\frac12-\frac{2p-1}{2(2p-3)}+\frac{1}{(2p-3)(2P-1)}\Big)\Big)\\
\ll \exp\Big(\frac{\log 2}{2}\log\log x-\frac{(\log x)^{1-\log2/2}}{2^{3/2} \log\log x}\Big)
=o(1).
\end{multline*}
For the third error term in (\ref{zasymp}), one checks 
%\begin{align*}
%x&^{-\frac{1}{2(2P-2)}}N^{\frac{2p-1}{2(2P-2)}-\frac{2}{P}}\log R\leq x^{-\frac{1}{2(2P-2)}} x^{\frac{1}{2p-3}\left(\frac{2p-1}{2(2P-2)}-\frac{2}{P}\right)}\log \big(x^{\frac{1}{(2p-3)(2P-1)}}\big)\\
%&=\frac{\log x}{(2p-3)(2P-1)}\exp\Big(-\log x\Big(\frac{1}{2(2P-2)}-\frac{2p-1}{2(2p-3)(2P-2)}+\frac{2}{(2p-3)P}\Big)\Big)\\
%&\leq\exp\Big(-\log x\Big(\frac{3P-4}{P(2P-2)(2p-3)}\Big)+\log\log x\Big)\\
%&\leq \exp\Big(-\frac{2\log x}{(2p-3)(2P-2)}+\log\log x\Big)\\
%&\leq \exp \Big(-\frac{(\log x)^{1-\log2/2}}{2^{1/2}\log\log x} +\log\log x\Big)=o(1).
%\end{align*}
\begin{multline*}
x^{-\frac{1}{2(2P-2)}}N^{\frac{2p-1}{2(2P-2)}-\frac{2}{P}}\log R
\ll \exp\Big(\log x\Big(\frac{-1}{2(2P-2)}+\frac{2p-1}{2(2p-3)(2P-2)}-\frac{2}{(2p-3)P}\Big)\Big) \log x\\
\ll \exp \Big(-\frac{(\log x)^{1-\log2/2}}{2^{1/2}\log\log x} +\log\log x\Big)
=o(1).
\end{multline*}
For the final error term in (\ref{zasymp}), one checks
%\begin{multline*}
%pN^{-1}\log R\leq p (x^{\frac1{2p-3}}-1)^{-1}\log \big(x^{\frac{1}{(2p-3)(2P-1)}}\big)\ll \frac{p\log x}{(2p-3)(2P-1)}\exp\Big(-\frac{\log x}{2p-3}\Big)\\
%\ll \exp\Big(-\frac{\log x}{\log\log x}+\log\log x\Big)=o(1).
%\end{multline*}
\begin{equation*}
pN^{-1}\log R \ll \exp\Big(-\frac{\log x}{\log\log x}+\log\log x\Big)=o(1).
\end{equation*}
The estimate (\ref{zasymp}) therefore shows \(\#Z(N)\sim N/20\) for the \(N\) given in (\ref{choices}), which satisfies \(N\leq \exp(\log x/(\log\log x-2))\). For this choice of \(N\), we thus conclude from (\ref{genbounddiag}) that
\[32\pi x^{1/2}\sum_{n\geq1} \frac{\Omega(n,x)^2}{n^{3/2}}\gg x^{1/2}N^{-1/2}\gg x^{1/2}\exp\Big(-\frac{\log x}{2(\log\log x-2)}\Big)\gg x^{1/2}\exp\Big(-\frac{\log x}{\log\log x}\Big),\]
as claimed.
\end{proof}

\begin{proof}[Proof of Lemma \ref{hdiscrepancylem}]
We apply Lemma \ref{erdosturanlemma}. This shows that for any positive integer \(R\),
\begin{equation}\label{erdosturan}
D(h, N)\ll\frac{N}{R}+\sum_{r\leq R}\frac{1}{r}\Big|\sum_{n\leq N}e(rh(n))\Big|.
\end{equation}
We will apply Lemma \ref{vdcorput} to bound the exponential sums \(\sum_{n\leq N} e(rh(n))\) appearing in (\ref{erdosturan}).
This requires estimates for the derivatives \(h^{(p)}\). 
One has \(h'(\xi)=\sqrt{2x}(1-\kappa^2/(32\pi ^2x\xi))^{1/2}\xi^{-1/2}\)  (see (\ref{hdef}) and (\ref{omega'})). 
By Taylor expanding and differentiating termwise, one computes that for \(p\geq1\) and \(\xi\geq1\), 
\begin{equation}\label{hpsauce}
h^{(p)}(\xi)
=
\sqrt{2x}
\cdot \frac{d^{p-1}}{d\xi^{p-1}} \big\{\xi^{-1/2}\big\}
\cdot \Big(1-S^p\Big(\frac{\kappa^2}{32\pi^2 x\xi}\Big)\Big),
\end{equation}
where for \(X<1\), 
\begin{equation*}
S^p(X)
=
\binom{2(p-1)}{p-1}^{-1}\sum_{l\geq1} \frac{1}{(2l-1)2^{2l}}\binom{p+l-1}{l}\binom{2(p+l-1)}{p+l-1} X^l .
\end{equation*}
Using our assumption \(x\geq k^2/(8\pi^2)\) and the simple bounds \(\frac{2^{2a}}{2\sqrt{a}}\leq \binom{2a}{a}\leq \frac{2^{2a}}{\sqrt{a}}\) (valid for any positive integer \(a\)), for \(\xi\geq 4p\) one can bound
\begin{equation}\label{spbound}
S^p\Big(\frac{\kappa^2}{32\pi^2 x\xi}\Big)
\leq 
2\sum_{l\geq1}\binom{p+l-1}{l}(4\xi)^{-l}
=
2\Big(\Big(1-\frac{1}{4\xi}\Big)^{-p}-1\Big)
\leq
\frac{2p}{\xi}
\leq 
\frac12.
\end{equation}
Additionally, Stirling's formula shows that for \(p\geq1\),
\begin{equation}\label{xideriv}
\frac{d^{p-1}}{d\xi^{p-1}} \big\{\xi^{-1/2}\big\}
=
\frac{(-1)^{p-1}(2(p-1))!}{2^{2(p-1)}(p-1)!} \xi^{1/2-p}
\asymp 
\Big(\frac{p}{e}\Big)^{p-1} \xi^{1/2-p}.
\end{equation}
Using (\ref{spbound}) and (\ref{xideriv}), we conclude from (\ref{hpsauce}) that for \(p\geq1\) and \(\xi\geq 4p\), 
\begin{equation}\label{hpsauced} 
h^{(p)}(\xi)\asymp x^{1/2}\frac{d^{p-1}}{d\xi^{p-1}} \big\{\xi^{-1/2}\big\} \asymp \Big(\frac pe\Big)^{p-1} x^{1/2}\xi^{1/2-p}.
\end{equation}

Now suppose \(p\geq2\), \(r\geq1\) and \(\xi\in[a,b]\), with \(a\geq 4p\) and \(b\leq 2a\).
It follows from (\ref{hpsauced}) that there exist positive (absolute) constants \(A_1\) and \(A_2\) such that
\[A_1 r \Big(\frac pe\Big)^{p-1} x^{1/2}b^{1/2-p}\leq r h^{(p)}(\xi)\leq A_2 2^p r \Big(\frac pe\Big)^{p-1} x^{1/2}b^{1/2-p}.\]
We apply Lemma \ref{vdcorput} with \(\lambda=A_1r(p/e)^{(p-1)} x^{1/2}b^{1/2-p}\) and \(\mu=A_2 2^p/A_1\). 
This gives
\begin{align*} 
\sum_{a<n\leq b}e(rh(n))&\ll b \Big(\frac{A_2 2^p}{A_1}\Big)^{2/P} \Big(A_1 r\Big(\frac{p}{e}\Big)^{p-1} x^{1/2}b^{1/2-p}\Big)^{1/(2P-2)}\\
&\hspace{14em}+b^{1-2/P} \Big(A_1 r \Big(\frac pe\Big)^{p-1} x^{1/2}b^{1/2-p}\Big)^{-1/(2P-2)}\\
&\ll x^{\frac{1}{2(2P-2)}}b^{1-\frac{2p-1}{2(2P-2)}}r^{\frac{1}{2P-2}}+x^{-\frac{1}{2(2P-2)}}b^{1+\frac{2p-1}{2(2P-2)}-\frac{2}{P}}r^{-\frac{1}{2P-2}}, 
\end{align*}
where \(P=2^{p-1}\). In the last line, we used that \(1/2\leq (p/e)^{(p-1)/(2^p-2)}\leq 2\) and \(1\leq 2^{2p/2^{p-1}}\leq 5\) for all \(p\geq 2\). Summing over dyadic intervals and bounding the contribution of \(n\leq 4p\) trivially, we obtain
\begin{equation*} 
\sum_{n\leq N}e(r h(n))\ll x^{\frac{1}{2(2P-2)}}N^{1-\frac{2p-1}{2(2P-2)}}r^{\frac{1}{2P-2}}+x^{-\frac{1}{2(2P-2)}}N^{1+\frac{2p-1}{2(2P-2)}-\frac{2}{P}}r^{-\frac{1}{2P-2}}+p.
\end{equation*}
The lemma now follows upon replacing this bound in (\ref{erdosturan}) and bounding the sums over \(r\).
%Returning to (\ref{erdosturan}) and applying the estimate above, for any integer \(R\geq 1\) we obtain
%\begin{multline*}
%D(h,N)\ll NR^{-1}+ x^{\frac{1}{2(2P-2)}}N^{1-\frac{2p-1}{2(2P-2)}}(2P-2)R^{\frac{1}{2P-2}}+x^{-\frac{1}{2(2P-2)}}N^{1+\frac{2p-1}{2(2P-2)}-\frac{2}{P}}\log R+p\log R\\
%\ll N\big( R^{-1}+2^p x^{\frac{1}{2(2P-2)}}N^{-\frac{2p-1}{2(2P-2)}}R^{\frac{1}{2P-2}} + x^{-\frac{1}{2(2P-2)}}N^{\frac{2p-1}{2(2P-2)}-\frac{2}{P}}\log R+pN^{-1}\log R\big),
%\end{multline*}
%as claimed.
\end{proof}

\subsection{The Off-diagonal Terms: Proof of Lemma \ref{offdiaglem}}\label{subsecoffdiag}

Our goal is to prove the following off-diagonal bound.

\offdiaglem*

We now fix \(0<\epsilon<1/1000\) for the remainder of this section (§\ref{subsecoffdiag}).
Recall the definition (\ref{offdiag}), which states
\begin{equation}\label{offdiagrep}
(\mathrm{OD})=8\pi^3(-1)^{k/2} x^2 \sum_{m,n\leq \Delta^2k^\epsilon/x} \tilde w\Big(\frac{nx}{k^2+\Delta^2}\Big)\tilde w\Big(\frac{mx}{k^2+\Delta^2}\Big) \sum_{c\geq1} c^{-1}S(m,n;c)J_{k-1}\Big(\frac{4\pi \sqrt{mn}}{c}\Big).
\end{equation}
Throughout this section, we will always assume \(x\geq k^2/(8\pi^2)\) and \(x^{1-\epsilon}\geq \Delta\geq x^{1/2}\). Moreover, considering the ranges of \(m\) and \(c\) in (\ref{offdiagrep}) above, we will also frequently assume \(c\geq1\) and \(m\leq \Delta^2k^\epsilon/x\). Several lemmas appearing in this section will depend on \(m\) and \(\Delta\), and these will hold uniformly for \(m\) and \(\Delta\) in these ranges (and for \(x\geq k^2/(8\pi^2)\)), unless otherwise stated. 

%As discussed in the introduction to Section \ref{secvar}, our strategy is roughly to interchange the order of summation in (\ref{offdiagrep}) and apply Poisson summation to the \(n\)-sum. This will capture some cancellation, coming from the Kloosterman sums, the oscillations of the smoothings \(\tilde w\), and the (in certain ranges of \(n,m\) and \(c\)) oscillations of the Bessel function \(J_{k-1}(4\pi\sqrt{mn}/c)\). 

In the following lemma, we perform some initial simplifications. We interchange the order of summation in (\ref{offdiagrep}) and smooth the resulting sums (with a view to applying Poisson summation later on).

\begin{lemma}\label{offdiag3lem}
Let \(g\) be a smooth, compactly supported function satisfying the following:
\begin{itemize}
\item{\(\supp(g)\subset[k/10,\Delta^2 k^\epsilon/x] \),}
\item{\(g(\xi)=1\) for \(k/2\leq\xi\leq \Delta^2 k^\epsilon/(2x)\),}
\item{for all integers \(j\geq 0\) and all \(\xi\), we have \(g^{(j)}(\xi)\ll_j \xi^{-j}\).}
\end{itemize}
We then have\footnote{For notational convenience we also implicitly set \(g(4\pi\sqrt{mn}/c)=0\) for any \(n<0\).}
\begin{multline*}
(\mathrm{OD})=8\pi^3(-1)^{k/2} x^2 \sum_{c\leq 100\Delta^2/(x k^{1-\epsilon})} c^{-1} \sum_{m\leq \Delta^2k^\epsilon/x} \tilde w\Big(\frac{mx}{k^2+\Delta^2}\Big)\\
\sum_{n\in\mathbb{Z}}S(m,n;c)\tilde w\Big(\frac{nx}{k^2+\Delta^2}\Big) g\Big(\frac{4\pi\sqrt{mn}}{c}\Big) J_{k-1}\Big(\frac{4\pi \sqrt{mn}}{c}\Big) +\mathcal O(k^{-1000}).
\end{multline*}
\end{lemma}

\begin{proof}
We first truncate the inner sum over \(c\) in (\ref{offdiagrep}) at \(100\Delta^2/(x k^{1-\epsilon})\). If \(c\geq 100\Delta^2/(x k^{1-\epsilon})\) then for any \(m,n\leq \Delta^2k^{\epsilon}/x\), the bound (\ref{bessboundsmallarg}) of Lemma \ref{bessellem} shows
\[J_{k-1}\Big(\frac{4\pi\sqrt{mn}}{c}\Big)\ll \frac{mn}{c^2}e^{-14k/13}.\]
%\[\frac{4\pi \sqrt{mn}}{c}\leq \frac{4\pi \Delta^2 k^\epsilon}{cx}\leq \frac{4\pi k}{100}\leq \frac k4 \implies J_{k-1}\Big(\frac{4\pi\sqrt{mn}}{c}\Big)\ll \frac{mn}{c^2}e^{-14k/13},\]
Using the trivial bound \(|S(m,n;c)|\leq c\) and the bound (\ref{w2}) of Lemma \ref{wgood} for the weights \(\tilde w\), the contribution of \(c\geq 100\Delta^2/(x k^{1-\epsilon})\) to (\ref{offdiagrep}) is thus seen to be \(\mathcal O(e^{-k})\). 
%\[\ll x^{1/2} \sum_{m,n\leq \Delta^2 k^\epsilon/x} (mn)^{1/4}\sum_{c\geq 100100\Delta^2/(x k^{1-\epsilon})} c^{-2} e^{-14k/13}\ll e^{-k}.\]
Truncating the sum over \(c\) and interchanging the order of summation, we thus rewrite (\ref{offdiagrep}) as 
\begin{multline}\label{offdiag2}
(\mathrm{OD})=8\pi^3(-1)^{k/2} x^2 \sum_{c\leq 100\Delta^2/(x k^{1-\epsilon})} c^{-1} \sum_{m\leq \Delta^2k^\epsilon/x} \tilde w\Big(\frac{mx}{k^2+\Delta^2}\Big)\\
\sum_{n\leq \Delta^2 k^\epsilon/x}S(m,n;c)\tilde w\Big(\frac{nx}{k^2+\Delta^2}\Big)J_{k-1}\Big(\frac{4\pi \sqrt{mn}}{c}\Big) +\mathcal O(e^{-k}).
\end{multline}

Finally, introducing the smoothing \(g(4\pi\sqrt{mn}/c)\) to the sums in (\ref{offdiag2}) produces an error that is \(\ll k^{-1000}\). 
Indeed, the contribution of those terms with \(4\pi\sqrt{mn}/c<k/2\) to (\ref{offdiag2}) is negligible, owing to the decay of the Bessel functions (see part \ref{besii} of Lemma \ref{bessellem}).
Also, since \(m\leq \Delta^2 k^\epsilon/x\), the contribution of those terms with \(4\pi\sqrt{mn}/c>\Delta^2 k^\epsilon/(2x)\) to (\ref{offdiag2}) is negligible, on account of the decay of the terms \(\tilde w(nx/(k^2+\Delta^2))\) for \(n\) in this range (see the bound (\ref{tildewibped}) of Lemma \ref{vorprop}).
The result follows. 
%Indeed, for these values of \(m,n\) and \(c\) the bound (\ref{bessel1}) shows \(J_{k-1}(4\pi \sqrt{mn}/c)\ll e^{-k^{1/2}}\), say, so their contribution is clearly \(\ll k^{-1000}\). On the other hand, since \(m\leq \Delta^2k^\epsilon/x\),
%\[\frac{4\pi\sqrt{mn}}{c}>\frac{2\Delta^2 k^\epsilon}{x}\implies n\geq \frac n {c^2}\geq\frac{1}{4\pi^2m}\Big(\frac{\Delta^2k^{\epsilon}}{x}\Big)^2\geq \frac{1}{4\pi^2}\frac{\Delta^2 k^\epsilon}{x}\implies  \frac{nx}{k^2+\Delta^2}\geq \frac{k^\epsilon}{4\pi^2(8\pi^2+1)},\]
%since we assume \(k^2\leq 8\pi^2 x\leq 8\pi^2\Delta^2\implies k^2+\Delta^2\leq (8\pi^2+1)\Delta^2\). The bound (\ref{tildewibped}) of Lemma \ref{vorprop} now shows
%\[\frac{4\pi\sqrt{mn}}{c}>\frac{2\Delta^2 k^\epsilon}{x}\implies \tilde w\Big(\frac{nx}{k^2+\Delta^2}\Big)\ll n^{-2}k^{-1100}.\]
%Thus the contribution of these terms to (\ref{offdiag2}) is indeed \(\ll k^{-1000}\).
\end{proof}

Our idea is now to replace \(\tilde w\) by the explicit expression given in Lemma \ref{wgood}. This leads to the following lemma, which is the key starting point for our proof of Lemma \ref{offdiaglem}.
\begin{lemma}\label{initalodbound}
We have the bound
\begin{equation*}
(\mathrm{OD})\ll x^{5/4}\sum_{c\leq 100\Delta^2/(x k^{1-\epsilon})} c^{-1}\sum_{m\leq \Delta^2k^\epsilon/x} m^{-3/4} \big(\max_{t\in[1,2]}|S_1|+\max_{t\in[1,2]}|S_2|\big)+x^{-5/4}k^{-4/3+2\epsilon}\Delta^{5/2},
\end{equation*}
where \(S_1=S_1(t;m,c)\) is given by 
\begin{equation}\label{s1'}
S_1
=\sum_{\substack{1\leq a\leq c\\ (a,c)=1}}\Big|\sum_{n\in\mathbb{Z}}(16\pi^2nxt-\kappa^2)^{-3/4}\sin \omega(4\pi \sqrt{nxt})g\Big(\frac{4\pi\sqrt{mn}}{c}\Big) J_{k-1}\Big(\frac{4\pi \sqrt{mn}}{c}\Big)e\Big(\frac{an}{c}\Big)\Big|,
\end{equation}
and \(S_2=S_2(t;m,c)\) is given by\footnote{The factor of \(x\) in the definition of \(S_2\) is included for convenience, so that \(S_1\) and \(S_2\) are of comparable sizes.}
\begin{equation}\label{s2'}
S_2
=x\sum_{\substack{1\leq a\leq c\\ (a,c)=1}}\Big|\sum_{n\in\mathbb{Z}}n(16\pi^2nxt-\kappa^2)^{-7/4}\sin \omega(4\pi \sqrt{nxt}) g\Big(\frac{4\pi\sqrt{mn}}{c}\Big) J_{k-1}\Big(\frac{4\pi \sqrt{mn}}{c}\Big)e\Big(\frac{an}{c}\Big)\Big|.
\end{equation}
\end{lemma}

\begin{proof}
Replacing \(\tilde w(nx/(k^2+\Delta^2))\) by the expression given in Lemma \ref{wgood} in Lemma \ref{offdiag3lem} and interchanging the order of summation and integration shows
\begin{multline}\label{offdiag4}
(\mathrm{OD})=2^{9/2}\pi^{5/2}(-1)^{k/2}x^2\sum_{c\leq 100\Delta^2/(x k^{1-\epsilon})} c^{-1}\sum_{m\leq \Delta^2k^\epsilon/x} \tilde w\Big(\frac{mx}{k^2+\Delta^2}\Big)\int_0^\infty \Big\{12\pi^2x tw(t)\\
\sum_{n_1\in\mathbb{Z}} S(m,n_1;c) n_1(16\pi^2n_1xt-\kappa^2)^{-7/4}\sin\omega(4\pi\sqrt{n_1xt})g\Big(\frac{4\pi\sqrt{mn_1}}{c}\Big) J_{k-1}\Big(\frac{4\pi \sqrt{mn_1}}{c}\Big) -(w(t)\\
+tw'(t))\sum_{n_2\in\mathbb{Z}} S(m,n_2;c)(16\pi^2n_2 xt-\kappa^2)^{-3/4}\sin\omega(4\pi\sqrt{n_2xt})g\Big(\frac{4\pi\sqrt{mn_2}}{c}\Big) J_{k-1}\Big(\frac{4\pi \sqrt{mn_2}}{c}\Big)\Big\}dt \\
+\mathcal O\Big(x^2\sum_{c\leq 100\Delta^2/(x k^{1-\epsilon})} c^{-1} \sum_{m\leq \Delta^2 k^\epsilon/x}\Big|\tilde w \Big(\frac{mx}{k^2+\Delta^2}\Big)\Big|
\\
\sum_{n\geq 1}(nx)^{-5/4}\Big|S(m,n;c) g\Big(\frac{4\pi\sqrt{mn}}{c}\Big) J_{k-1}\Big(\frac{4\pi \sqrt{mn}}{c}\Big)\Big|\Big)+\mathcal O(k^{-1000}) .
\end{multline}

To bound the error term above, we use the following bounds. From Lemma \ref{wgood}, we have \(|\tilde w(mx/(k^2+\Delta^2))|\ll (mx)^{-3/4}\). Trivially \(|S(m,n;c)|\leq c\). By the construction of \(g\), we have \(|g(\xi)|\ll 1\) for all \(\xi\). Finally, we use the Bessel function bound (\ref{bessel2}) of Lemma \ref{bessellem}, which gives \(|J_{k-1}(4\pi\sqrt{mn}/c)|\ll k^{-1/3}\). These bounds show that the error in (\ref{offdiag4}) is 
\[\ll k^{-1/3} \sum_{c\leq 100\Delta^2/(x k^{1-\epsilon})}\sum_{m\leq \Delta^2 k^\epsilon/x} m^{-3/4}\sum_{n\geq1} n^{-5/4}\ll x^{-5/4}k^{-4/3+2\epsilon}\Delta^{5/2}.\]

Opening the Kloosterman sums in (\ref{offdiag4}) and interchanging the order of summation, we now obtain 
\begin{multline}\label{offdiag4v2}
(\mathrm{OD})=2^{9/2}\pi^{5/2}(-1)^{k/2}x^2\sum_{c\leq 100\Delta^2/(x k^{1-\epsilon})} c^{-1} \sum_{m\leq \Delta^2k^\epsilon/x} \tilde w\Big(\frac{mx}{k^2+\Delta^2}\Big)\\
\int_0^\infty \Big\{12\pi^2x tw(t) \sum_{\substack{1\leq a_1\leq c\\(a_1,c)=1}} e\Big(\frac{a_1^*m}c\Big)\sum_{n_1\in\mathbb{Z}} n_1(16\pi^2n_1xt-\kappa^2)^{-7/4}\sin\omega(4\pi\sqrt{n_1xt})g\Big(\frac{4\pi\sqrt{mn_1}}{c}\Big) \\
J_{k-1}\Big(\frac{4\pi \sqrt{mn_1}}{c}\Big)e\Big(\frac{a_1n_1}c\Big) -(w(t)+tw'(t))\sum_{\substack{1\leq a_2\leq c\\(a_2,c)=1}} e\Big(\frac{a_2^*m}c\Big)\sum_{n_2\in\mathbb{Z}} (16\pi^2n_2xt-\kappa^2)^{-3/4}\\
\sin\omega(4\pi\sqrt{n_2xt})g\Big(\frac{4\pi\sqrt{mn_2}}{c}\Big) J_{k-1}\Big(\frac{4\pi \sqrt{mn_2}}{c}\Big)e\Big(\frac{a_2n_2}c\Big)\Big\}dt 
+\mathcal O(x^{-5/4}k^{-4/3+2\epsilon}\Delta^{5/2}).
\end{multline}
Using the bound \(\tilde w(mx/(k^2+\Delta^2))\ll (mx)^{-3/4}\)
(which is (\ref{w2}) of Lemma \ref{wgood}), the triangle inequality shows
\begin{multline}\label{offdiag4v3}
(\mathrm{OD})\ll x^{5/4}\sum_{c\leq 100\Delta^2/(x k^{1-\epsilon})} c^{-1} \sum_{m\leq \Delta^2k^\epsilon/x} m^{-3/4}\\
\int_0^\infty \big\{|tw(t)| |S_2(t;m,c)|+|w(t)+tw'(t)||S_1(t;m,c)|\big\}dt+x^{-5/4}k^{-4/3+2\epsilon}\Delta^{5/2}\\
\ll x^{5/4}\sum_{c\leq 100\Delta^2/(x k^{1-\epsilon})} c^{-1} \sum_{m\leq \Delta^2k^\epsilon/x} m^{-3/4}\Big(\max_{t\in[1,2]}|S_2|\int_0^\infty |tw(t)|dt\\
+\max_{t\in[1,2]}|S_1|\int_0^\infty |w(t)+tw'(t)|dt\Big)
+x^{-5/4}k^{-4/3+2\epsilon}\Delta^{5/2}.
\end{multline}
In the last line, we observed that the range of integration is \(t\in\supp w=[1,2]\). Recall also (from Definition \ref{wdef}) that \(w\ll 1\), \(\supp w'=[1, 1+\Delta^{-1}]\cup [2-\Delta^{-2}, 2]\) and \(w'\ll \Delta\). It follows
\[\int_0^\infty |tw(t)|dt\ll 1 \text{ and } \int_0^\infty |w(t)+tw'(t)|dt\ll 1.\]
The lemma now follows from (\ref{offdiag4v3}). 
\end{proof}

\begin{remark}
Considering \(t\in[1,2]\) in Lemma \ref{initalodbound}, from now on we implicitly assume \(t\in[1,2]\) wherever this variable \(t\) appears. Indeed, the following lemmas hold uniformly for \(t\in[1,2]\). 
\end{remark}

We wish to bound the sums \(S_1\) and \(S_2\). To do so, we first apply Poisson summation, which gives the following.

\begin{lemma}[Poisson Summation]\label{poissonlem2}
For \(i=1,2\) we have the bounds 
\begin{equation*}
S_i\ll c^{1/2}m^{-1/4}\sum_{\substack{n\geq0 \\ (n,c)=1}}|\mathcal I_i(n)|.
\end{equation*}
Here
\[\mathcal I_i(n)=\int_0^\infty G_i(y) \sin\omega\Big(\frac{c\sqrt{xt}}{\sqrt{m}}y\Big)e\Big(\frac{cn}{16\pi^2 m}y^2\Big)dy,\]
where
\begin{equation}\label{G1}
G_1(y)=c^{3/2}m^{-3/4}y\Big(\frac{c^2xt}{m}y^2-\kappa^2\Big)^{-3/4}g(y)J_{\kappa}(y), 
\end{equation}
and 
\begin{equation}\label{G2}
G_2(y)=c^{7/2}m^{-7/4}x y^3\Big(\frac{c^2xt}{m}y^2-\kappa^2\Big)^{-7/4} g(y)J_{\kappa}(y).
\end{equation}
\end{lemma}

\begin{proof}
We consider only \(S_1\) here, since \(S_2\) may be handled in exactly the same way. Splitting into progressions, write
\begin{multline}\label{residuesplits}
S_1=\sum_{\substack{1\leq a\leq c\\ (a,c)=1}}\Big|\sum_{n\in\mathbb{Z}}(16\pi^2nxt-\kappa^2)^{-3/4}\sin \omega(4\pi \sqrt{nxt})g\Big(\frac{4\pi\sqrt{mn}}{c}\Big) J_{k-1}\Big(\frac{4\pi \sqrt{mn}}{c}\Big)e\Big(\frac{an}{c}\Big)\Big|\\
=\sum_{\substack{1\leq a\leq c\\ (a,c)=1}}\Big|\sum_{b\Mod c} e\Big(\frac{ab}{c}\Big)\sum_{\substack{l\in\mathbb{Z}}}(16\pi^2(b+lc)xt-\kappa^2)^{-3/4}\sin \omega(4\pi \sqrt{(b+lc)xt})\\
g\Big(\frac{4\pi\sqrt{m(b+lc)}}{c}\Big) J_{k-1}\Big(\frac{4\pi \sqrt{m(b+lc)}}{c}\Big)\Big|.
\end{multline}
Applying Poisson summation to the inner sum over \(l\), we find
\begin{align*}
&\sum_{\substack{l\in\mathbb{Z}}}(16\pi^2(b+lc)xt-\kappa^2)^{-3/4}\sin \omega(4\pi \sqrt{(b+lc)xt})
g\Big(\frac{4\pi\sqrt{m(b+lc)}}{c}\Big) J_{k-1}\Big(\frac{4\pi \sqrt{m(b+lc)}}{c}\Big)\\
=&\sum_{\tilde l\in\mathbb{Z}}\int_0^\infty (16\pi^2(b+vc)xt-\kappa^2)^{-3/4}\sin \omega(4\pi \sqrt{(b+vc)xt})\\
&\hspace{10em} g\Big(\frac{4\pi\sqrt{m(b+vc)}}{c}\Big) J_{k-1}\Big(\frac{4\pi \sqrt{m(b+vc)}}{c}\Big)e(-\tilde lv)dv\\
=&\sum_{\tilde l\in\mathbb{Z}}\int_0^\infty (16\pi^2 uxt-\kappa^2)^{-3/4}\sin \omega(4\pi \sqrt{uxt})
g\Big(\frac{4\pi\sqrt{mu}}{c}\Big) J_{k-1}\Big(\frac{4\pi \sqrt{mu}}{c}\Big)e\Big(-\frac{\tilde l(u-b)}{c}\Big)\frac{du}{c}.
\end{align*}
(In the last line, we have set \(u=b+vc\).) Consequently, we obtain from (\ref{residuesplits})
\begin{align}\label{triangleineqbounds}
S_1&=\sum_{\substack{1\leq a\leq c\\ (a,c)=1}}\Big|\sum_{\tilde l\in \mathbb{Z}}\int_0^\infty (16\pi^2 uxt-\kappa^2)^{-3/4}\sin \omega(4\pi \sqrt{uxt})
g\Big(\frac{4\pi\sqrt{mu}}{c}\Big)J_{k-1}\Big(\frac{4\pi \sqrt{mu}}{c}\Big)\\
&\hspace{20em} e\Big(-\frac{\tilde lu}{c}\Big) \Big\{\frac{1}{c}\sum_{b\Mod c} e\Big(\frac{ b(\tilde l+a)}{c}\Big)\Big\} du\Big|\nonumber\\
&=\sum_{\substack{1\leq a\leq c\\ (a,c)=1}}\Big|\sum_{\substack{\tilde l\in \mathbb{Z}\\ \tilde l\equiv -a\Mod c}}\int_0^\infty (16\pi^2 uxt-\kappa^2)^{-3/4}\sin \omega(4\pi \sqrt{uxt})
g\Big(\frac{4\pi\sqrt{mu}}{c}\Big)
\nonumber\\
&\hspace{24.15em} J_{k-1}\Big(\frac{4\pi \sqrt{mu}}{c}\Big)e\Big(-\frac{\tilde lu}{c}\Big) du\Big|\nonumber\\
&\leq \sum_{\substack{n\in\mathbb{Z}\\ (n,c)=1}}\Big|\int_0^\infty (16\pi^2 uxt-\kappa^2)^{-3/4}\sin \omega(4\pi \sqrt{uxt})
g\Big(\frac{4\pi\sqrt{mu}}{c}\Big) J_{k-1}\Big(\frac{4\pi \sqrt{mu}}{c}\Big)e\Big(-\frac{nu}{c}\Big) du\Big|\nonumber.
\end{align}
The last line is a consequence of the triangle inequality. Setting \(y=4\pi\sqrt{mu}/c\), we find
\begin{multline*}
\int_0^\infty (16\pi^2 uxt-\kappa^2)^{-3/4}\sin \omega(4\pi \sqrt{uxt})
g\Big(\frac{4\pi\sqrt{mu}}{c}\Big) J_{k-1}\Big(\frac{4\pi \sqrt{mu}}{c}\Big)
e\Big(-\frac{nu}{c}\Big) du\\
=\frac{c^{1/2}}{8\pi^2m^{1/4}}\mathcal I_1(-n).
\end{multline*}
Noting \(\mathcal I_1(-n)=\overline{\mathcal I_1(n)}\), the result now follows from (\ref{triangleineqbounds}). 
\end{proof}

\begin{remark} 
Since the integrals \(\mathcal I_i(n)\) are taken over \(y\in\supp g=[k/10, \Delta^2 k^\epsilon/x]\), from now on we will assume \(k/10\leq y\leq \Delta^2 k^\epsilon/x\) wherever this variable appears.
\end{remark}
Our goal is now to bound the integrals \(\mathcal I_i(n)\) appearing in Lemma \ref{poissonlem2} via the method of stationary phase. Before proceeding further, in the following lemma we first handle some very simple cases where \(\mathcal I_i(n)\) is negligibly small. After some minor aesthetic rearrangements, we thus truncate the sum over \(n\) in Lemma \ref{poissonlem2}.

\begin{lemma} \label{s12'lem}
Assume \(\Delta\leq x^{2/3}k^{1/3-\epsilon}\). For \(i=1,2\) we have
\begin{equation}\label{s12'}
S_i\ll c^{1/2}m^{-1/4}\sum_{\frac{m^{1/2}x^{3/2}}{\Delta^2 k^\epsilon}\leq n\leq \frac{1000\sqrt{mx}}{k}}|\mathcal I_i'(n)|+k^{-1000},
\end{equation}
where
\begin{equation}\label{i'def}
\mathcal I_i'(n)=\int_0^\infty G_i(y)e^{iF(y)}dy.
\end{equation}
Here \(G_1\) and \(G_2\) are the functions given in Lemma \ref{poissonlem2}, and 
\begin{equation}\label{F}
F(y)=\frac{cn}{8\pi m}y^2-\omega\Big(\frac{c\sqrt{xt}}{\sqrt{m}}y\Big).
\end{equation}
\end{lemma}

\begin{proof}
Starting from Lemma \ref{poissonlem2}, we first note 
\begin{multline}\label{cosreplaced}
\mathcal I_i(n)=\int_0^\infty G_i(y)\sin\omega\Big(\frac{c\sqrt{xt}}{\sqrt{m}}y\Big)e\Big(\frac{cn}{16\pi^2 m}y^2\Big)dy\\
=\frac{1}{2i}\int_0^\infty G_i(y)\exp\Big\{i\Big(\frac{cn}{8\pi m}y^2+\omega\Big(\frac{c\sqrt{xt}}{\sqrt{m}}y\Big)\Big)\Big\}dy\\
-\frac{1}{2i}\int_0^\infty G_i(y)\exp\Big\{i\Big(\frac{cn}{8\pi m}y^2-\omega\Big(\frac{c\sqrt{xt}}{\sqrt{m}}y\Big)\Big)\Big\}dy.
\end{multline}
Using Lemma \ref{bkyipb}, we will show that the integrals above contribute \(\ll k^{-1000}\) to \(S_i\) in many cases. To do so, we must bound the derivatives of the oscillatory phases in (\ref{cosreplaced}). One computes (using (\ref{omega'}))
\begin{equation}\label{firstderiv0}
\frac{d}{dy}\Big\{\frac{cn}{8\pi m}y^2\pm \omega\Big(\frac{c\sqrt{xt}}{\sqrt{m}}y\Big)\Big\}=\frac{cn}{4\pi m} y\pm \Big(\frac{c^2xt}{m}-\frac{\kappa^2}{y^2}\Big)^{1/2},
\end{equation}
and (using (\ref{omega''}))
\begin{equation}\label{secondderiv0}
\frac{d^2}{dy^2}\Big\{\frac{cn}{8\pi m}y^2\pm \omega\Big(\frac{c\sqrt{xt}}{\sqrt{m}}y\Big)\Big\}=\frac{cn}{4\pi m}\pm\frac{\kappa^2}{y^3} \Big(\frac{c^2xt}{m}-\frac{\kappa^2}{y^2}\Big)^{-1/2}.
\end{equation}
Recall we assume \(\Delta\leq x^{2/3}k^{1/3-\epsilon}\). 
From this, using also the standing assumptions that \(x\geq k^2/(8\pi^2)\), \(c\geq1\), \(1\leq t\leq 2\), \(m\leq \Delta^2k^\epsilon/x\) and \(k/10\leq y\leq \Delta^2k^\epsilon/x\), one has
\begin{equation}\label{asympgoodone}
\frac{c^2xt}{m}
%\gg \frac{x^2}{\Delta^2k^\epsilon} 
%\gg x^{2/3}k^{-2/3+\epsilon}
\gg k^{2/3}\text{ and } \frac{\kappa^2}{y^2}\ll 1\implies \Big(\frac{c^2xt}{m}-\frac{\kappa^2}{y^2}\Big)
%\sim \frac{c^2xt}{m}
\asymp \frac{c^2 x}{m}.
\end{equation}
In particular, using \(m\leq \Delta^2 k^\epsilon/x\), \(\Delta\leq x^{2/3}k^{1/3-\epsilon}\), \(y\geq k/10\) and \(c\geq1\), we obtain for \(n\neq0\)
\begin{equation}\label{littleonicer}
\frac{d^2}{dy^2}\Big\{\omega\Big(\frac{c\sqrt{xt}}{\sqrt{m}}y\Big)\Big\}=\frac{\kappa^2}{y^3} \Big(\frac{c^2xt}{m}-\frac{\kappa^2}{y^2}\Big)^{-1/2}\asymp \frac{k^2 m^{1/2}}{y^3 c x^{1/2}}
%\ll
%\frac{\Delta}{xk^{1-\epsilon}}
%\ll
%\frac{x}{\Delta^2 k^\epsilon}
=o\Big(\frac{cn}{4\pi m}\Big).
\end{equation}
%\[\frac{\kappa^2}{y^3} \Big(\frac{c^2xt}{m}-\frac{\kappa^2}{y^2}\Big)^{-1/2}\asymp \frac{k^2 m^{1/2}}{y^3 c x^{1/2}}\ll \frac{\Delta}{x k^{1-\epsilon/2}}.\]
%Now the assumption \(\Delta\leq x^{2/3}k^{1/3-\epsilon}\) (and \(m\leq \Delta^2 k^\epsilon/x\)) guarantees
%\begin{equation*}
%\frac{\Delta}{x k^{1-\epsilon/2}}\leq \frac{x}{\Delta^2k^\epsilon} \cdot k^{-3\epsilon/2}\ll \frac{1}{m}\cdot k^{-3\epsilon/2}=o\Big(\frac 1m\Big).
%\end{equation*}
%Combining the two calculations above, we have shown (provided \(n\neq 0\))
%\begin{equation}
%\frac{d^2}{dy^2}\Big\{\omega\Big(\frac{c\sqrt{xt}}{\sqrt{m}}y\Big)\Big\}=\frac{\kappa^2}{y^3} \Big(\frac{c^2xt}{m}-\frac{\kappa^2}{y^2}\Big)^{-1/2}\asymp \frac{k^2 m^{1/2}}{y^3 c x^{1/2}}=o\Big(\frac{cn}{4\pi m}\Big).
%\end{equation}
Consequently, for \(n\neq0\) we have established the second derivative estimate
\begin{equation}\label{secondderivasymp}
\frac{d^2}{dy^2}\Big\{\frac{cn}{8\pi m}y^2\pm \omega\Big(\frac{c\sqrt{xt}}{\sqrt{m}}y\Big)\Big\}\asymp \frac{cn}{m}.
\end{equation}

We next bound the higher derivatives. For \(y\geq k/10\), one differentiates repeatedly (using (\ref{asympgoodone})) to see that for any \(a>0\) and any integer \(j\geq 1\),
%\begin{equation}\label{omegaderivatives} 
%\frac{d^j}{dy^j} \Big\{\Big(\frac{c^2xt}{m}-\frac{\kappa^2}{y^2}\Big)^{-a}\Big\}
%\ll_j
%\Big(\frac{c^2 xt}{m}\Big)^{-a-1}\frac{k^2}{y^{3+j}}.
%\end{equation}
\begin{equation}\label{omegaderivatives}
\frac{d^j}{dy^j} \Big\{\Big(\frac{c^2xt}{m}-\frac{\kappa^2}{y^2}\Big)^{-a}\Big\}
\ll_{j,a}
\frac{\kappa^2}{y^{2+j}}\Big(\frac{c^2xt}{m}-\frac{\kappa^2}{y^2}\Big)^{-a-1}
\ll_{j,a} 
y^{-j}\Big(\frac{c^2xt}{m}-\frac{\kappa^2}{y^2}\Big)^{-a}.
\end{equation}
So for any \(j\geq2\) and \(y\geq k/10\), we differentiate repeatedly, applying (\ref{littleonicer}) and (\ref{omegaderivatives}) to obtain the bound
\begin{multline}\label{phasederivbound}
\frac{d^j}{dy^j}\Big\{\frac{cn}{8\pi m}y^2\pm\omega\Big(\frac{c\sqrt{xt}}{\sqrt{m}}y\Big)\Big\}
\ll
\frac{cn}{m}\delta_{j=2}
+\Big|\frac{d^{j-2}}{dy^{j-2}} \Big\{\frac{\kappa^2}{y^3}\Big(\frac{c^2xt}{m}-\frac{\kappa^2}{y^2}\Big)^{-1/2}\Big\}\Big|\\
\ll_j
\frac{cn}{m}\delta_{j=2} 
+\frac{1}{y^{j-2}}\cdot \frac{\kappa^2}{y^3}\Big(\frac{c^2xt}{m}-\frac{\kappa^2}{y^2}\Big)^{-1/2}
\ll_j
\frac{cn}{m} y^{2-j}
\ll_j
\frac{cn}{m} k^{2-j}.
\end{multline}

%We now claim that for \(n\neq0\) and \(k/10\leq y\leq \Delta^2k^\epsilon/x\) in the region of integration, we have 
%\begin{equation}\label{phasederivbound}
%\frac{d^j}{dy^j}\Big\{\frac{cn}{8\pi m}y^2\pm\omega\Big(\frac{c\sqrt{xt}}{\sqrt{m}}y\Big)\Big\}\ll_j \frac{cn}{m}y^{2-j} \ll_j\frac{cn}{m}k^{2-j} \text{ for } j=2,3,\ldots.
%\end{equation}
%Indeed, the \(j=2\) case is included in (\ref{secondderivasymp}). For \(y\geq k/10\), one has (using (\ref{asympgoodone})) that for any \(a>0\) and an integer \(j\geq 1\),
%\begin{equation}\label{omegaderivatives} 
%\frac{d^j}{dy^j} \Big\{\Big(\frac{c^2xt}{m}-\frac{\kappa^2}{y^2}\Big)^{-a}\Big\}\ll_j\Big(\frac{c^2 xt}{m}\Big)^{-a-1}\frac{k^2}{y^{3+j}}.
%\end{equation}
%Thus, upon each differentiation of 
%\[\frac{d^2}{dy^2}\Big\{\omega\Big(\frac{c\sqrt{xt}}{\sqrt{m}}y\Big)\Big\}=\frac{\kappa^2}{y^3} \Big(\frac{c^2xt}{m}-\frac{\kappa^2}{y^2}\Big)^{-1/2},\]
%one saves at least \(y\). That is to say, for \(j>2\)
%\begin{equation*}
%\frac{d^j}{dy^j}\Big\{\frac{cn}{8\pi m}y^2\pm\omega\Big(\frac{c\sqrt{xt}}{\sqrt{m}}y\Big)\Big\}=\pm \frac{d^j}{dy^j}\Big\{\omega\Big(\frac{c\sqrt{xt}}{\sqrt{m}}y\Big)\Big\}\ll_j y^{-(j-2)}\frac{d^2}{dy^2}\Big\{\omega\Big(\frac{c\sqrt{xt}}{\sqrt{m}}y\Big)\Big\}\ll_j \frac{cn}{m} y^{2-j},
%\end{equation*}
%(using (\ref{littleonicer})) as claimed.

Finally, we must bound the derivatives of \(G_1\) and \(G_2\). Recall (from (\ref{G1}) and (\ref{G2}))
\begin{equation*}
G_1(y)=c^{3/2}m^{-3/4}y^{-1/2}\Big(\frac{c^2xt}{m}-\frac{\kappa^2}{y^2}\Big)^{-3/4}g(y)J_{\kappa}(y), 
\end{equation*}
and 
\begin{equation*}
G_2(y)=c^{7/2}m^{-7/4}x y^{-1/2}\Big(\frac{c^2xt}{m}-\frac{\kappa^2}{y^2}\Big)^{-7/4} g(y)J_{\kappa}(y).
\end{equation*}
Recall also \(g^{(j)}(y)\ll_j y^{-j}\) (by our construction, see Lemma \ref{offdiag3lem}). For any \(\nu\geq0\), one has
\begin{equation}\label{besseldiff2}
J_\nu'(y)=\frac{1}{2}(J_{\nu-1}(y)-J_{\nu+1}(y)).
\end{equation}
This is easily checked from the definition (\ref{besseldef}) (see also \cite[§3.2]{watson}). Using the derivative bound (\ref{omegaderivatives}), and using (\ref{besseldiff2}) to differentiate the Bessel functions (\(\leq j\) times), we obtain the bound
\begin{equation} \label{argderivbound}
G_i^{(j)}(y)\ll_j x^{-3/4}y^{-1/2}(y^{-j}|J_{\kappa}(y)|+|J_{\kappa+\mathcal O_j(1)} (y)|)\ll_j x^{-3/4}k^{-5/6} \text{ for } j=0,1,2\ldots
\end{equation}
Here we also used the bound \(J_{\kappa+\mathcal O(1)}(y)\ll k^{-1/3}\) given in (\ref{bessel2}) (and our assumption \(y\geq k/10\)).

Using these calculations, we now apply Lemma \ref{bkyipb} to show that the integrals appearing in (\ref{cosreplaced}) are negligible (i.e. contribute \(\ll k^{-1000}\)) in the following four simple cases.

\begin{enumerate}[wide, label=(\roman*)]
\item \label{ibpi}
Firstly, consider the term \(\mathcal I_i(0)\) (which appears in Lemma \ref{poissonlem2} only if \(c=1\)). Here, it follows from (\ref{asympgoodone}) that the phase satisfies
\begin{equation}\label{firstderiv0case}
\frac{d}{dy}\Big\{\pm \omega\Big(\frac{\sqrt{xt}}{\sqrt{m}}y\Big)\Big\}=\pm \Big(\frac{xt}{m}-\frac{\kappa^2}{y^2}\Big)^{1/2}\gg \Big(\frac{x}{m}\Big)^{1/2}.
\end{equation}
Turning our attention to the second derivative, the \(c=1\) case of (\ref{littleonicer}) gives that for \(y\geq k/10\)
\[\frac{d^2}{dy^2}\Big\{\pm \omega\Big(\frac{\sqrt{xt}}{\sqrt{m}}y\Big)\Big\}
=
\pm\frac{\kappa^2}{y^3}\Big(\frac{xt}{m}-\frac{\kappa^2}{y^2}\Big)^{-1/2}
\asymp 
\frac{k^2m^{1/2}}{y^3 x^{1/2}}
\ll 
\frac{1}{k}\Big(\frac{m}{x}\Big)^{1/2}.\]
Differentiating further, using (\ref{omegaderivatives}) and the above, we have for \(j\geq2\) and \(y\geq k/10\) that
\begin{equation}\label{manyderiv0case}
\frac{d^j}{dy^j}\Big\{\pm \omega\Big(\frac{\sqrt{xt}}{\sqrt{m}}y\Big)\Big\}
=
\frac{d^{j-2}}{dy^{j-2}}\Big\{\pm\frac{\kappa^2}{y^3}\Big(\frac{xt}{m}-\frac{\kappa^2}{y^2}\Big)^{-1/2}\Big\}
\ll_j
y^{-(j-2)} \frac{1}{k}\Big(\frac{m}{x}\Big)^{1/2}
\ll_j 
k^{1-j}\Big(\frac{m}{x}\Big)^{1/2}.
\end{equation}
Consequently, using the estimates (\ref{argderivbound}), (\ref{firstderiv0case}) and (\ref{manyderiv0case}), we may apply Lemma \ref{bkyipb} with \(\alpha=k/10\), \(\beta=\Delta^2k^\epsilon/x\), \(X=x^{-3/4}k^{-5/6}\), \(U=1\), \(R=(x/m)^{1/2}\), \(Y=k(m/x)^{1/2}\) and \(Q=k\). This gives the bound
\begin{equation}\label{ibpapp1}
\mathcal I_i(0)\ll_B \Delta^2x^{-7/4}k^{-5/6+\epsilon}\Big\{\Big(\Big(\frac{x}{m}\Big)^{3/4}k^{1/2}\Big)^{-B}+\Big(\frac{x}{m}\Big)^{-B/2}\Big\},
\end{equation}
valid for any integer \(B\geq0\). 
Together with \(m\leq \Delta^2 k^\epsilon/x\), our assumption \(\Delta\leq x^{2/3}k^{1/3-\epsilon}\) implies \(1/m\geq x^{-1/3}k^{-2/3+\epsilon}\). In particular, our assumption \(x\geq k^2/(8\pi^2)\) implies
\begin{equation}\label{mbound} 
\frac{x}{m}\geq \Big(\frac{x}{k}\Big)^{2/3} k^\epsilon \geq k^{2/3}.
\end{equation}
So taking \(B\) large enough in (\ref{ibpapp1}), we easily obtain \(\mathcal I_i(0)\ll k^{-1000}\).

\item\label{ibpii}
Next, we consider the former integral appearing in (\ref{cosreplaced}):
\[\int_0^\infty G_i(y)\exp\Big\{i\Big(\frac{cn}{8\pi m}y^2+\omega\Big(\frac{c\sqrt{xt}}{\sqrt{m}}y\Big)\Big)\Big\}dy.\]
In this case one has for \(y\geq k/10\) (see (\ref{firstderiv0}) and (\ref{asympgoodone})),
\begin{equation}\label{firstderivplusbound}
\frac{d}{dy}\Big\{\frac{cn}{8\pi m}y^2+ \omega\Big(\frac{c\sqrt{xt}}{\sqrt{m}}y\Big)\Big\}=\frac{cn}{4\pi m} y+ \Big(\frac{c^2xt}{m}-\frac{\kappa^2}{y^2}\Big)^{1/2}\gg \frac{cnk}{m}+\frac{c x^{1/2}}{m^{1/2}}.
\end{equation}
Equipped with (\ref{phasederivbound}), (\ref{argderivbound}) and (\ref{firstderivplusbound}), we apply Lemma \ref{bkyipb} with \(\alpha=k/10\), \(\beta=\Delta^2k^\epsilon/x\), \(X=x^{-3/4}k^{-5/6}\), \(U=1\), \(R=cnk/m+c\sqrt{x/m}\), \(Y=cnk^2/m\) and \(Q=k\). This shows
\begin{multline}\label{ibpapp2}
\int_0^\infty G_i(y)\exp\Big\{i\Big(\frac{cn}{8\pi m}y^2+\omega\Big(\frac{c\sqrt{xt}}{\sqrt{m}}y\Big)\Big)\Big\}dy\\
\ll_B \Delta^2 x^{-7/4}k^{-5/6+\epsilon}\Big\{\Big(\frac{cnk}{m}+\frac{c\sqrt{x}}{\sqrt{m}}\Big)^{-B}+\Big(\frac{\sqrt{cn}k}{\sqrt{m}}+\frac{\sqrt{cx}}{\sqrt{n}}\Big)^{-B}\Big\},
\end{multline}
valid for any integer \(B\geq0\). We will use the bound
\begin{equation*} 
\Big(\frac{cnk}{m}+\frac{c\sqrt{x}}{\sqrt{m}}\Big)^{-B}+\Big(\frac{\sqrt{cn}k}{\sqrt{m}}+\frac{\sqrt{cx}}{\sqrt{n}}\Big)^{-B}\leq \Big(\frac{nk}{m}\Big)^{-2} \Big(\frac xm\Big)^{-(B-2)/2}+\Big(\frac{nk^2}{m}\Big)^{-2}\Big(\frac{nk^2}{m}+\frac{x}{n}\Big)^{-(B-4)/2}.
\end{equation*}
For any \(n\), we have 
\[\frac{nk^2}{m}+\frac{x}{n}\geq 2\frac{kx^{1/2}}{m^{1/2}}\geq k^{4/3},\]
since \(x/m\geq k^{2/3}\) by (\ref{mbound}). So we obtain
\[\Big(\frac{cnk}{m}+\frac{c\sqrt{x}}{\sqrt{m}}\Big)^{-B}+\Big(\frac{\sqrt{cn}k}{\sqrt{m}}+\frac{\sqrt{cx}}{\sqrt{n}}\Big)^{-B} \leq n^{-2}\Big( \frac{m^2}{k^2}\cdot k^{-(B-2)/3}+ \frac{m^2}{k^4}k^{-2(B-4)/3}\Big).\]
Thus taking \(B\) large enough in (\ref{ibpapp2}) shows
\[\int_0^\infty G_i(y)\exp\Big\{i\Big(\frac{cn}{8\pi m}y^2+\omega\Big(\frac{c\sqrt{xt}}{\sqrt{m}}y\Big)\Big)\Big\}dy \ll n^{-2}k^{-1100},\]
say, and consequently these integrals contribute \(\ll k^{-1000}\) to \(S_i\).

\item\label{ibpiii}
Finally, we consider the latter integral in (\ref{cosreplaced}): 
\[\int_0^\infty G_i(y)\exp\Big\{i\Big(\frac{cn}{8\pi m}y^2-\omega\Big(\frac{c\sqrt{xt}}{\sqrt{m}}y\Big)\Big)\Big\}dy\coloneqq \mathcal I_i'(n).\]
This is more difficult than the previous case, since now the derivative of the oscillatory phase \(F(y)\) in \(\mathcal I_i'(n)\) can vanish:
\[F'(y)=\frac{d}{dy}\Big\{\frac{cn}{8\pi m}y^2- \omega\Big(\frac{c\sqrt{xt}}{\sqrt{m}}y\Big)\Big\}=\frac{cn}{4\pi m} y- \Big(\frac{c^2xt}{m}-\frac{\kappa^2}{y^2}\Big)^{1/2}.\]
But clearly if \(n\) is large enough, \(F'(y)\) must be also be large. Indeed, suppose \(n\geq 1000\sqrt{mx}/k\). Then for \(1\leq t\leq 2\) and \(y\geq k/10\) (in the range of integration), one has
\begin{equation*} 
F'(y)=\frac{cn}{4\pi m} y- \Big(\frac{c^2xt}{m}-\frac{\kappa^2}{y^2}\Big)^{1/2}\geq \frac{cnk}{40\pi m}-\frac{c\sqrt{xt}}{\sqrt{m}} \geq \frac{cnk}{m}\Big(\frac{1}{40\pi}-\frac{\sqrt{2}}{1000}\Big) \gg \frac{cnk}{m}.
\end{equation*}
Using this estimate (and also (\ref{phasederivbound}) and (\ref{argderivbound})), we may apply Lemma \ref{bkyipb} with \(\alpha=k/10\), \(\beta=\Delta^2k^\epsilon/x\), \(X=x^{-3/4}k^{-5/6}\), \(U=1\), \(R= cnk/m\), \(Y=cnk^2/m\) and \(Q=k\).
We obtain
\begin{multline}\label{ibpapp3} 
\mathcal I_i'(n)=\int_0^\infty G_i(y)\exp\Big\{i\Big(\frac{cn}{8\pi m}y^2-\omega\Big(\frac{c\sqrt{xt}}{\sqrt{m}}y\Big)\Big)\Big\}dy\\
\ll_B \Delta^2x^{-7/4}k^{-5/6+\epsilon}\Big\{\Big(\frac{cnk^2}{m}\Big)^{-B/2}+\Big(\frac{cnk}{m}\Big)^{-B}\Big\},
\end{multline}
valid for any integer \(B\geq0\). Using (\ref{mbound}) (which states \(x/m\geq k^{2/3}\)) one has for \(n\geq 1000\sqrt{mx}/k\) that \(cnk^2/m\geq cnk/m\gg cx^{1/2}/m^{1/2}\gg k^{1/3}\). Taking \(B\) large enough, we concludes from (\ref{ibpapp3}) that \(\mathcal I_i'(n)\ll n^{-2}k^{-1100}\). Thus these \(\mathcal I_i'(n)\) with \(n\geq 1000\sqrt{mx}/k\) contribute \(\ll k^{-1000}\) to \(S_i\).

\item\label{ibpiv} 
Finally, we show that the integrals \(\mathcal I_i'(n)\) are also negligible when \(n\) is very small. Indeed, if \(n\leq m^{1/2}x^{3/2}/(\Delta^2k^\epsilon)\) then using (\ref{asympgoodone}) and our assumption \(y\leq \Delta^2 k^\epsilon/x\) we obtain
\begin{equation*}
|F'(y)|\geq \Big|\Big(\frac{c^2xt}{m}-\frac{\kappa^2}{y^2}\Big)^{1/2}-\frac{cn}{4\pi m} y\Big|
\geq \frac12 \frac{cx^{1/2}t^{1/2}}{m^{1/2}}-\frac{cn}{4\pi m}\frac{\Delta^2k^\epsilon}{x}\geq \frac{cx^{1/2}}{m^{1/2}} \Big(\frac12-\frac{1}{4\pi}\Big)\gg c\Big(\frac{x}{m}\Big)^{1/2}.
\end{equation*}
Using this estimate together with (\ref{phasederivbound}) and (\ref{argderivbound}) as before, we are able to apply Lemma \ref{bkyipb} with \(\alpha=k/10\), \(\beta=\Delta^2k^\epsilon/x\), \(X=x^{-3/4}k^{-5/6}\), \(U=1\), \(R=c\sqrt{x/m}\), \(Y=cnk^2/m\) and \(Q=k\). This provides the bound
\begin{equation}\label{i'boundlastone}
\mathcal I_i'(n)\ll_B \Delta^2 x^{-7/4}k^{-5/6+\epsilon}\Big\{\Big(\frac{cx}{n}\Big)^{-B/2}+\Big(\frac {c^2x}{m}\Big)^{-B/2}\Big\},
\end{equation}
valid for any integer \(B\geq0\). Since \(m\leq \Delta^2 k^\epsilon/x\), we have
\[n\leq \frac{m^{1/2}x^{3/2}}{\Delta^2 k^{\epsilon}}\implies \frac{cx}{n}\geq 
%\frac{\Delta^2 k^{\epsilon}}{m^{1/2}x^{1/2}}\geq 
\Delta k^{\epsilon/2}.\]
Because \(c^2x/m\geq k^{2/3}\) is also large (see (\ref{mbound})), taking \(B\) sufficiently large in (\ref{i'boundlastone}) shows that those \(\mathcal I_i'(n)\) with \(n\leq m^{1/2}x^{3/2}/(\Delta^2k^\epsilon)\) contribute \(\ll k^{-1000}\) to \(S_i\).
\end{enumerate}
\end{proof}

\begin{remark}
We record the following useful facts from the above proof. Firstly, we established in (\ref{argderivbound}) that \(G_1\) and \(G_2\) satisfy
\begin{equation} \label{argderivboundrep}
G_i^{(j)}(y)\ll_j x^{-3/4}k^{-5/6} \text{ for } j=0,1,2\ldots
\end{equation}
Considering the phase \(F(y)\), we showed (see (\ref{firstderiv0}) and (\ref{secondderiv0}))
\begin{equation}\label{F'}
F'(y)=\frac{cn}{4\pi m} y- \Big(\frac{c^2xt}{m}-\frac{\kappa^2}{y^2}\Big)^{1/2},
\end{equation}
and
\begin{equation}\label{F''really}
F''(y)=\frac{cn}{4\pi m}-\frac{\kappa^2}{y^3} \Big(\frac{c^2xt}{m}-\frac{\kappa^2}{y^2}\Big)^{-1/2}.
\end{equation}
Finally, we showed (see (\ref{secondderivasymp}) and (\ref{phasederivbound})) that under the condition \(\Delta \leq x^{2/3}k^{1/3-\epsilon}\), for \(k/10\leq y\leq \Delta^2 k^\epsilon/x\) in the region of integration 
\begin{equation}\label{Fderivs}
F''(y)\asymp \frac{cn}{m}, \text{ and } F^{(j)}(y)\ll_j \frac{cn}{m} y^{2-j} \text{ for } j=2,3,\ldots.
\end{equation}
\end{remark}

It remains to bound the contribution of the \(\mathcal I_i'(n)\) with \(m^{1/2}x^{3/2}/(\Delta^2 k^\epsilon)\leq n\leq 1000\sqrt{mx}/k\). These integrals are not negligible, since in this case a larger contribution appears from the stationary phase. The analysis is therefore more involved.

Let \(y_0=y_0(n)\) be the stationary phase of \(F\) in the region of integration, i.e. \(y_0\) satisfies \(F'(y_0)=0\) and \(y_0\in [k/10, \Delta^2k^\epsilon/x].\) 
Solving \(F'(y_0)=0\) (see (\ref{F'})) shows
\begin{equation*} 
y_0^2
=
\frac{8\pi^2mxt}{n^2}\Big(1\pm \Big(1-\frac{\kappa^2n^2}{4\pi^2c^2x^2t^2}\Big)^{1/2}\Big)
=
\frac{8\pi^2mxt}{n^2}\Big(1\pm 1 +\mathcal O\Big(\frac{k^2 n^2}{x^2}\Big)\Big).
\end{equation*}
(The last step follows from a Taylor expansion and the facts \(1\leq t\leq 2\) and \(c\geq1\).)
Only one solution (corresponding to \(\pm=+\)) in the above is possible. 
Indeed, our assumptions \(m\leq \Delta^2 k^\epsilon/x\) and \(\Delta\leq x^{1-\epsilon}\) imply 
\begin{equation*}
\frac{8\pi^2mxt}{n^2}\Big(1- 1 +\mathcal O\Big(\frac{k^2 n^2}{x^2}\Big)\Big)
\ll
\frac{k^2 m}{x} =o(k^2).
\end{equation*}
But we take \(y_0\in[k/10, \Delta^2 k^\epsilon/x]\), so \(y_0^2\neq o(k^2)\). Thus
\[y_0^2=\frac{16\pi^2 mxt}{n^2}+\mathcal O\Big(\frac{k^2 m}{x}\Big) \iff y_0=\frac{4\pi\sqrt{mxt}}{n}+\mathcal O\Big(\frac{k^2nm^{1/2}}{x^{3/2}}\Big).\]
%From (\ref{F'}), one finds
%\begin{equation*} 
%y_0^2=\frac{1}{2}\Big\{\frac{16\pi^2 mxt}{n^2}\pm \Big(\Big(\frac{16\pi^2 mxt}{n^2}\Big)^2-\frac{64\pi^2 m^2\kappa^2}{c^2 n^2}\Big)^{1/2}\Big\}=\frac{8\pi^2mxt}{n^2}\Big(1\pm \Big(1-\frac{\kappa^2n^2}{4\pi^2c^2x^2t^2}\Big)^{1/2}\Big).
%\end{equation*}
%Since \(m\leq \Delta^2 k^\epsilon /x\) and \(1\leq t\leq 2\), if \(n\leq 1000\sqrt{mx}/k\) then 
%\[\frac{\kappa^2n^2}{4\pi^2c^2x^2t^2} \leq \frac{10^6}{4\pi^2}\frac{ m}{c^2 x t^2}\ll \frac{\Delta^2 k^\epsilon}{x^2}=o(1),\]
%under the standing assumption \(\Delta\leq x^{1-\epsilon}\), say. So a Taylor expansion shows
%\[\Big(1-\frac{\kappa^2n^2}{4\pi^2c^2x^2t^2}\Big)^{1/2} = 1+ \mathcal O\Big(\frac{k^2 n^2}{c^2 x^2}\Big)\implies y_0^2= \frac{16\pi^2 mxt}{n^2}+\mathcal O\Big(\frac{k^2 m}{c^2 x}\Big), \text{ or } y_0^2=\mathcal O\Big(\frac{k^2 m}{c^2 x}\Big).\]
%But using \(m\leq \Delta^2 k^\epsilon/x\) and \(\Delta\leq x^{1-\epsilon}\), we observe
%\[\frac{k^2m}{c^2x}\leq \frac{\Delta^2 k^{2+\epsilon}}{x^2}\leq k^{2+\epsilon}x^{-2\epsilon}< \Big(\frac{k}{10}\Big)^2.\]
%So \(y_0^2=\mathcal O(k^2m/(c^2x))\) is impossible, since we assume \(y_0\geq k/10\) (so that \(y_0\) lies in the region of integration). It follows
%\[y_0^2=\frac{16\pi^2 mxt}{n^2}+\mathcal O\Big(\frac{k^2 m}{c^2 x}\Big) \iff y_0=\frac{4\pi\sqrt{mxt}}{n}+\mathcal O\Big(\frac{k^2nm^{1/2}}{c^2x^{3/2}}\Big).\]
We bound the error term above by using that \(m\leq \Delta^2 k^\epsilon/x\), and the assumption \(n\leq 1000\sqrt{mx}/k\). From this, we conclude
\begin{equation}\label{y0def}
y_0= \frac{4\pi\sqrt{mxt}}{n}+\mathcal O\Big(\frac{k^{1+\epsilon}\Delta^2}{x^2}\Big).
\end{equation}

We will show that the only significant contribution to the integrals \(\mathcal I_i'(n)\) comes from a neighbourhood of \(y_0\). 
To do so, we use a smooth partition of unity to split the range of integration in \(\mathcal I_i'(n)\) into intervals surrounding the stationary point \(y_0\).
From now on, set \(L=L(n)=k^\epsilon\max\{1,m/(cn)\}\), and let \((b_l^{L,y_0})_{l\in\mathbb{Z}}\) be the smooth functions given in Lemma \ref{smthpart}.
Then
\begin{multline}\label{smthpartapplied}
\mathcal I_i'(n)=\int_0^\infty \Big(\sum_{l\in\mathbb{Z}}b_l^{L,y_0}(y)\Big)G_i(y)e^{iF(y)}dy=\int_{y_0-2L}^{y_0+2L}b_0^{L,y_0}(y)G_i(y)e^{iF(y)}dy\\
+\sum_{l\geq 1} \int_{y_0+2^{l-1}L}^{y_0+2^{l+1}L}b_l^{L,y_0}(y)G_i(y)e^{iF(y)}dy+\sum_{l\leq-1}\int_{y_0-2^{|l|+1}L}^{y_0-2^{|l|-1}L} b_l^{L,y_0}(y)G_i(y)e^{iF(y)}dy.
\end{multline}
In the following lemma, we show that all but the \(l=0\) term is negligible.

\begin{lemma}\label{Ii'nlem}
Let \(L\) be as above, and assume \(\Delta\leq x^{2/3}k^{1/3-\epsilon}\). Then for \(i=1,2\) we have
\begin{equation*}
\mathcal I_i'(n)=\int_{y_0-2L}^{y_0+2L}b_0^{L,y_0}(y)G_i(y)e^{iF(y)}dy+\mathcal O(k^{-1000}).
\end{equation*}
\end{lemma}

\begin{proof}
For \(i=1,2\), we apply Lemma \ref{bkyipb} to bound the \(l\neq0\) terms appearing in (\ref{smthpartapplied}), which are:
\[\int_0^\infty b_l^{L,y_0}(y)G_i(y)e^{iF(y)}dy.\]
From the estimate \(\frac{d^j}{dy^j} b_l^{L,y_0}(y)\ll_j 2^{-j|l|}L^{-j}\) (given in Lemma \ref{smthpart}) and (\ref{argderivboundrep}), we deduce 
\begin{equation}\label{argderivs2}
\frac{d^j}{dy^j}\{b_l^{L,y_0}(y)G_i(y)\}\ll_j x^{-3/4}k^{-5/6}, \text{ for } j=0,1,2,\ldots
\end{equation}
From (\ref{Fderivs}), we also have 
\begin{equation}\label{phasederivs2}
F^{(j)}(y)\ll_j \frac{cn}{m}y^{2-j} \text{ for } j=2,3,\ldots, \text{ and } F''(y)\asymp \frac{cn}{m}\implies F'(y)\asymp \frac{cn}{m}|y-y_0|.
\end{equation}
Equipped with (\ref{argderivs2}) and (\ref{phasederivs2}), and noting that \(|y-y_0|\asymp 2^{|l|}L \) for \(y\in\supp b_l^{L,y_0}\) (provided \(l\neq0\)), we are ready to apply Lemma \ref{bkyipb}. We take \([\alpha, \beta]\) to be the interval \(\supp b_l^{L,y_0}\) (so that \(\beta-\alpha \asymp 2^{|l|}L\)), and take \(X=x^{-3/4}k^{-5/6}\), \(U=1\), \(R=cn 2^{|l|}L/m\), \(Y=cn\alpha^2/m\) and \(Q=\alpha\) (where \(\alpha\asymp \beta\) is the infimum of \(\supp b_l^{L,y_0}\)). This yields
\begin{equation*}
\int_0^\infty b_l^{L,y_0}(y)G_i(y)e^{iF(y)}dy\ll_B 2^{|l|}L x^{-3/4}k^{-5/6} \Big\{2^{-|l|B}\Big(\Big(\frac{cn}{m}L^2\Big)^{-B/2}+\Big(\frac{cn}{m}L\Big)^{-B}\Big)\Big\},
\end{equation*}
valid for any non-negative integer \(B\). But with the choice of \(L=k^\epsilon\max\{1, m/(cn)\}\), one always has \(cnL^2/m>cnL/m\geq k^\epsilon\). So by taking \(B\) large enough, the contribution of the integrals with \(l\neq0\) to (\ref{smthpartapplied}) is seen to be \(\ll k^{-1000}\).
\end{proof}

\begin{remark}
From now on, we will denote \(b=b_0^{L,y_0}\) for simplicity.
\end{remark}

Our final task is to bound the integrals appearing in Lemma \ref{Ii'nlem}. We consider three cases separately, based on the value of \(n\), or equivalently \(y_0(n)\). The first (and easiest) is the case where the region of integration \([y_0(n)-2L,y_0(n)+2L]\) lies in the region in which the Bessel function \(J_{k-1}(y)\) is negligibly small. Secondly, we consider the case in which \(y_0(n)\) is close to the transition of the Bessel function. The final case is that in which the region of integration lies within the oscillatory regime of the Bessel function.

First recall \(L=k^\epsilon\max\{1,m/(cn)\}\). For \(n\geq m^{1/2}x^{3/2}/(\Delta^2k^\epsilon)\) in the range of summation of Lemma \ref{s12'lem}, using \(m\leq \Delta^2 k^\epsilon/x\) and assuming \(\Delta\leq x^{2/3}k^{1/3-\epsilon}\) we find
\begin{equation}\label{Lllk}
L=L(n)\leq k^\epsilon\Big(1+\frac{m}{cn}\Big)\ll k^{1-\epsilon/2}.
\end{equation}
We also remark that assuming \(\Delta\leq x^{2/3}k^{1/3-\epsilon}\) (and using \(x\geq k^2/(8\pi^2)\)), the error term in (\ref{y0def}) satisfies \(k^{1+\epsilon}\Delta^2 /x^2 \ll k^{1/3-\epsilon}\). Thus for \(\Delta\leq x^{2/3}k^{1/3-\epsilon}\) and \(n\leq 1000\sqrt{mx}/k\), (\ref{y0def}) shows
\begin{equation}\label{y0exp}
y_0=y_0(n)=\frac{4\pi\sqrt{mxt}}{n}+\mathcal O(k^{1/3-\epsilon})\sim \frac{4\pi \sqrt{mxt}}{n}\gg k.
\end{equation}
In particular, combining (\ref{Lllk}) and (\ref{y0exp}) shows that for \(y_0-2L\leq y\leq y_0+2L\) in the region of integration of \(\mathcal I_i'(n)\), one has \(y=y_0+\mathcal O(L)\asymp y_0\).

\begin{lemma}\label{blomerexplem}
Assume \(\Delta\leq x^{2/3}k^{1/3-\epsilon}\). For \(m^{1/2}x^{3/2}/(\Delta^2 k^\epsilon)\leq n\leq 1000\sqrt{mx}/k\) and \(i=1,2\) we have the following bounds.

\begin{enumerate}[label=(\roman*)]
\item \label{iini} 
If \(y_0(n)+2L\leq k-k^{1/3+\epsilon}\), then \(\mathcal I_i'(n)\ll e^{-k^\epsilon}\).

\item \label{iinii} 
If \(y_0(n)+2L\geq k-k^{1/3+\epsilon}\) and \(y_0(n)-10L\leq k+k^{1/3+\epsilon}\), then
\begin{equation*}
\mathcal I_i'(n)\ll x^{-3/4}k^{-5/6+\epsilon}+x^{-5/4}k^{1/6+\epsilon}c^{-1}m^{1/2}.
\end{equation*}

\item \label{iiniii}
If \(y_0(n)-10L\geq k+k^{1/3+\epsilon}\), then
\begin{equation*}
\mathcal I_i'(n)\ll x^{-9/8}k^{-1/4}c^{-1/2}m^{1/8}n^{1/2}\Big(\frac{4\pi\sqrt{mxt}}{k}-n\Big)^{-1/4}+x^{-5/8}k^{-9/4}m^{1/8}n^2\Big(\frac{4\pi\sqrt{mxt}}{k}-n\Big)^{-9/4}.
\end{equation*}
\end{enumerate}
\end{lemma}

\begin{proof}
Parts \ref{iini} and \ref{iinii} are relatively simple. For both of these, we use the trivial bound for the integrals \(\mathcal I_i'(n)\) (see the expression given in Lemma \ref{Ii'nlem}), which shows
\begin{equation}\label{trivialboundIi'n} 
\mathcal I_i'(n)\leq \int_{y_0-2L}^{y_0+2L} |b(y) G_i(y)|dy +\mathcal O(k^{-1000}) \ll L \max_{y\in[y_0-2L,y_0+2L]} |G_i(y)|+k^{-1000}.
\end{equation}
It follows from the definitions (\ref{G1}) and (\ref{G2}) (cf. (\ref{argderivboundrep})) that for \(i=1,2\) and \(y\in[y_0-2L, y_0+2L]\) (these \(y\) satisfy \(y\asymp y_0\gg k\), see (\ref{Lllk}) and (\ref{y0exp})),
\begin{equation}\label{simpleGibound}
G_i(y)\ll x^{-3/4}k^{-1/2}|J_{k-1}(y)|.
\end{equation}

For part \ref{iini}, (\ref{bessel1}) shows \(J_{k-1}(y)\ll e^{-k^\epsilon}\) for \(y\leq y_0+2L\leq k-k^{1/3+\epsilon}\). So we conclude from (\ref{trivialboundIi'n}) and (\ref{simpleGibound}) that
\[\mathcal I_i'(n)\ll L x^{-3/4} k^{-1/2} e^{-k^\epsilon}\ll e^{-k^\epsilon},\]
as claimed.

For part \ref{iinii}, we instead use the bound \(J_{k-1}(y)\ll k^{-1/3}\) given in (\ref{bessel2}) (which is sharp in the transition regime). In this case, (\ref{trivialboundIi'n}) and (\ref{simpleGibound}) therefore show
\begin{equation}\label{trivialboundIi'}
\mathcal I_i'(n)\ll x^{-3/4}k^{-5/6}L\ll x^{-3/4}k^{-5/6+\epsilon}\Big(1+\frac{m}{cn}\Big).
\end{equation}
Using that \(y_0\asymp \sqrt{mx}/n\) from (\ref{y0exp}) and \(L\ll k^{1-\epsilon/2}=o(y_0)\) from (\ref{Lllk}), we have
\[y_0-10L\leq k+k^{1/3+\epsilon}\implies \frac{\sqrt{mx}}{n}\ll k \implies \frac{m}{cn}\ll \frac{km^{1/2}}{cx^{1/2}}.\]
So the required bound follows from (\ref{trivialboundIi'}) and the preceding estimate. 

Part \ref{iiniii} corresponds to the case where the region of integration is entirely contained in the oscillatory regime of the Bessel function \(J_{k-1}(y)\), and is considerably more involved. We first replace the Bessel functions by the asymptotic (\ref{bessel3}) of Lemma \ref{bessellem}, which is available in this regime. For \(y\geq k+k^{1/3}\), (\ref{bessel3}) gives
\begin{multline}\label{truncatedbessasymp}
J_{k-1}(y)=
\frac{1}{\sqrt{2\pi}}(y^2-\kappa^2)^{-1/4}\Bigg\{\Big(1+\frac{1}{8i}(y^2-\kappa^2)^{-1/2}+\frac{5}{24i}\kappa^2(y^2-\kappa^2)^{-3/2}\Big)e^{i\omega(y)}\\
+\Big(1-\frac{1}{8i}(y^2-\kappa^2)^{-1/2}-\frac{5}{24i}\kappa^2(y^2-\kappa^2)^{-3/2}\Big)e^{-i\omega(y)}\Bigg\}+\mathcal O\Big(\frac{y^4}{(y^2-k^2)^{13/4}}\Big).
\end{multline}
%&\hspace{0.7em}+\Big(1-\frac{1}{8i}(y^2-\kappa^2)^{-1/2}-\frac{5}{24i}\kappa^2(y^2-\kappa^2)^{-3/2}\Big)e^{-i\omega(y)}\Bigg\}+\mathcal O\Big(\frac{y^4}{(y^2-k^2)^{13/4}}\Big).
%\begin{align}\label{truncatedbessasymp}
%J_{k-1}(y)&=\sqrt{\frac{2}{\pi}}(y^2-\kappa^2)^{-1/4}\cos \omega(y)+\sqrt{\frac{2}{\pi}}(y^2-\kappa^2)^{-3/4}\Big(\frac{1}{8}+\frac{5}{24}\frac{\kappa^2}{y^2-\kappa^2}\Big)\sin \omega(y)\\
%&\hspace{24.5em}+\mathcal O\Big(\frac{y^4}{(y^2-\kappa^2)^{13/4}}\Big)\nonumber\\
%&=\frac{1}{\sqrt{2\pi}}(y^2-\kappa^2)^{-1/4}\Bigg\{\Big(1+\frac{1}{8i}(y^2-\kappa^2)^{-1/2}+\frac{5}{24i}\kappa^2(y^2-\kappa^2)^{-3/2}\Big)e^{i\omega(y)}\nonumber\\
%&\hspace{0.7em}+\Big(1-\frac{1}{8i}(y^2-\kappa^2)^{-1/2}-\frac{5}{24i}\kappa^2(y^2-\kappa^2)^{-3/2}\Big)e^{-i\omega(y)}\Bigg\}+\mathcal O\Big(\frac{y^4}{(y^2-k^2)^{13/4}}\Big).\nonumber
%\end{align}
Replacing this in the expressions (\ref{G1}) and (\ref{G2}) for \(G_1\) and \(G_2\) (given in Lemma \ref{poissonlem2}), we have from Lemma \ref{Ii'nlem} that
\begin{multline}\label{Ii'nexp3}
\mathcal I_i'(n)= \int_{y_0-2L}^{y_0+2L}b(y)G_i(y)e^{iF(y)}dy+\mathcal O(k^{-1000})\\
=\int_{y_0-2L}^{y_0+2L} G_{i+}(y)e^{i(F(y)+\omega(y))}dy+\int_{y_0-2L}^{y_0+2L} G_{i-}(y)e^{i(F(y)-\omega(y))}dy\\
+\mathcal O\Big(\int_{y_0-2L}^{y_0+2L}|b(y)g(y) x^{-3/4} y^{-1/2}|\cdot y^4 (y^2-k^2)^{-13/4}dy\Big)+\mathcal O(k^{-1000}),
\end{multline}
where the functions \(G_{i\pm}\) are given by
\begin{align*}
G_{1\pm}(y)&=\frac{1}{\sqrt{2\pi}}c^{3/2}m^{-3/4}b(y)g(y)y\Big(\frac{c^2xt}{m}y^2-\kappa^2\Big)^{-3/4}(y^2-\kappa^2)^{-1/4}\\
&\hspace{16em}\Big(1\pm\frac{1}{8i}(y^2-\kappa^2)^{-1/2}\pm\frac{5}{24i}\kappa^2(y^2-\kappa^2)^{-3/2}\Big),\\
G_{2\pm}(y)&=\frac{1}{\sqrt{2\pi}}c^{7/2}m^{-7/4}xb(y)g(y)y^3\Big(\frac{c^2xt}{m}y^2-\kappa^2\Big)^{-7/4}(y^2-\kappa^2)^{-1/4}\\
&\hspace{16em}\Big(1\pm\frac{1}{8i}(y^2-\kappa^2)^{-1/2}\pm\frac{5}{24i}\kappa^2(y^2-\kappa^2)^{-3/2}\Big).
\end{align*}

We now bound the error term in (\ref{Ii'nexp3}). Since the smoothing functions \(b\) and \(g\) are bounded (by their construction), this is
\begin{multline}\label{errorboundmess} 
\int_{y_0-2L}^{y_0+2L}|b(y)g(y) x^{-3/4}y^{-1/2} |\cdot y^4 (y^2-k^2)^{-13/4}dy\\
\ll x^{-3/4} \int_{y_0-2L}^{y_0+2L} y^{7/2}(y+k)^{-13/4}(y-k)^{-13/4}dy
%\ll x^{-3/4}(y_0+2L)^{7/2} (y_0-2L+k)^{-13/4} \int_{y_0-2L}^{y_0+2L} (y-k)^{-13/4}dy\\
\ll x^{-3/4}y_0^{1/4} (y_0-2L-k)^{-9/4}.
\end{multline}
In the last step, we used that for \(y\in[y_0-2L, y_0+2L]\) and \(n\) in the range \(m^{1/2}x^{3/2}/(\Delta^2 k^\epsilon)\leq n\leq 1000\sqrt{mx}/k\), we have \(y\asymp y+k\asymp y_0\) (see (\ref{Lllk}) and (\ref{y0exp})). Additionally, (\ref{y0exp}) shows
\[\Big|(y_0-2L-k)-\Big(\frac{4\pi\sqrt{mxt}}{n}-k\Big)\Big|\leq 2L +\Big|y_0-\frac{4\pi\sqrt{mxt}}{n}\Big|\leq 2L+\mathcal O(k^{1/3-\epsilon}).\]
Since for part \ref{iiniii} we assume \(y_0-2L-k\geq 8L +k^{1/3+\epsilon}\), it follows \((y_0-2L-k)\asymp (4\pi\sqrt{mxt}/n-k)\). Consequently (\ref{errorboundmess}) is 
\[\ll x^{-3/4}y_0^{1/4}\Big(\frac{4\pi\sqrt{mxt}}{n}-k\Big)^{-9/4}\ll x^{-5/8}k^{-9/4}m^{1/8}n^2\Big(\frac{4\pi\sqrt{mxt}}{k}-n\Big)^{-9/4}.\]
(We also used \(y_0\asymp \sqrt{mx}/n\), see (\ref{y0exp}).) Replacing this error bound in (\ref{Ii'nexp3}), we have established
\begin{multline}\label{Ii'nexp4} 
\mathcal I_i'(n)= \int_{y_0-2L}^{y_0+2L} G_{i+}(y)e^{i(F(y)+\omega(y))}dy+\int_{y_0-2L}^{y_0+2L} G_{i-}(y)e^{i(F(y)-\omega(y))}dy\\
+\mathcal O\Big(x^{-5/8}k^{-9/4}m^{1/8}n^2\Big(\frac{4\pi\sqrt{mxt}}{k}-n\Big)^{-9/4}\Big).
\end{multline}

The final task is to bound the oscillatory integrals appearing in (\ref{Ii'nexp4}). This will be done using Lemma \ref{bkyexpansion}. In order to apply this, we first require a bound for the derivatives of the functions \(G_{i\pm}\). This follows relatively straightforwardly from the (above) definitions. Indeed, firstly one recalls that for any \(j\geq0\), we have \(b^{(j)}(y)\ll_j L^{-j}\) and \(g^{(j)}(y)\ll_j y^{-j}\) (by construction). Secondly, if \(y\in[y_0-2L, y_0+2L]\) then \(y\asymp y_0\) (see (\ref{Lllk}) and (\ref{y0exp})). For these \(y\), the bound \(c^2xt/m\geq x/m\geq k^{2/3}\) (valid for \(m\leq \Delta^2k^\epsilon/x\), \(c\geq1\), \(1\leq t\leq 2\) and \(\Delta\leq x^{2/3}k^{1/3-\epsilon}\), see (\ref{mbound})) also shows \((c^2xty^2/m-\kappa^2)\asymp c^2x y_0^2/m\). For \(y\in[y_0-2L, y_0+2L]\) and \(j\geq0\), one now computes
\begin{multline*} 
\frac{d^j}{dy^j} \Big(\frac{c^2 xt}{m} y^2-\kappa^2\Big)^{-3/4}\ll_j y^{-j}\Big(\frac{c^2 xt}{m} y^2-\kappa^2\Big)^{-3/4}\ll_j \Big(\frac{c^2 x}{m}\Big)^{-3/4} y_0^{-3/2-j},\\ 
\text{ and similarly } \frac{d^j}{dy^j} \Big(\frac{c^2 xt}{m} y^2-\kappa^2\Big)^{-7/4}\ll_j y^{-j}\Big(\frac{c^2 xt}{m} y^2-\kappa^2\Big)^{-7/4}\ll_j \Big(\frac{c^2x}{m}\Big)^{-7/4} y_0^{-7/2-j}.
\end{multline*}
Finally, for \(y\in[y_0-2L, y_0+2L]\) and \(j\geq0\) we have
\[\frac{d^j}{dy^j} (y^2-\kappa^2)^{-1/4} 
%\ll_j \Big(\frac{y}{y^2-\kappa^2}\Big)^{j}(y^2-\kappa^2)^{-1/4}
\ll_j y_0^{-1/4}(y_0-k)^{-1/4-j}.\]
Here we used that \(y-\kappa \asymp y_0-k\), which follows from our assumption \(y_0-k\geq 10L +k^{1/3+\epsilon}\) and the fact that \(|y-y_0|\leq 2L\) in this range of \(y\). The above calculations show that upon differentiating \(G_{i\pm}\), we save at least \(L\) each time (note \(L\ll y_0-k\) by our assumption \(y_0-k\geq 10L+k^{1/3+\epsilon}\)). In other words (assuming \(\Delta\leq x^{2/3}k^{1/3-\epsilon}\)) we obtain for \(y\in[y_0-2L,y_0+2L]\) and integers \(j\geq0\) the bound 
\begin{equation}\label{tildegderivs}
\frac{d^j}{dy^j}G_{i\pm}(y)\ll_j x^{-3/4}y_0^{-3/4}(y_0-k)^{-1/4}L^{-j}.
\end{equation}

We next turn our attention to the new oscillatory phase. Set
\[F_\pm(y)=F(y)\pm\omega(y).\]
It turns out that despite the extra \(\omega(y)\) term, \(F_\pm\) behaves essentially the same as \(F\) on the interval \([y_0-2L,y_0+2L]\). Indeed, by (\ref{omega'}) one has
\begin{equation}\label{tildeF'}
F_\pm'(y)=F'(y)\pm\Big(1-\frac{\kappa^2}{y^2}\Big)^{1/2}=F'(y)+\mathcal O(1),
\end{equation}
Furthermore, by (\ref{omega''})
\begin{equation}\label{tildeF''} 
F_\pm''(y)=F''(y)\pm\frac{\kappa^2}{y^2}(y^2-\kappa^2)^{-1/2}.
\end{equation}
Recall from (\ref{Fderivs}) that \(F''(y)\asymp cn/m\). 
Recall also (\ref{y0exp}), which gives that for \(n\leq 1000\sqrt{mx}/k\), \(y_0\asymp \sqrt{mx}/n\gg k\). 
Note also that for \(y\) in the region of integration, \(y-k\geq y_0-2L-k\geq k^{1/3+\epsilon}\) by assumption.
Combining these observations with our assumptions \(x\geq k^2/(8\pi^2)\), \(m\leq \Delta^2k^\epsilon/x\) and \(\Delta\leq x^{2/3}k^{1/3-\epsilon}\), one can show that for \(y\in[y_0-2L, y_0+2L]\), 
%It therefore follows from (\ref{y0exp}) (which gives \(y_0\asymp \sqrt{mx}/n\)) that \(F''(y)\asymp cx^{1/2}/(y_0 m^{1/2})\). 
%On the other hand, if \(y\) is in the region of integration, then \(y-k\geq y_0-2L-k\geq k^{1/3+\epsilon}\) by assumption. 
%Using the assumptions \(x\geq k^2/(8\pi^2)\), \(m\leq \Delta^2k^\epsilon/x\) and \(\Delta\leq x^{2/3}k^{1/3-\epsilon}\), one can then show that for these \(y\),
\begin{equation}\label{kappaF''bound}
\frac{\kappa^2}{y^2}(y^2-\kappa^2)^{-1/2}\ll \frac{k^{11/6-\epsilon/2}}{y_0^{5/2}}
%=o\Big(\frac{x^{1/2}}{y_0 m^{1/2}}\Big)
=o\Big( \frac{cn}{m}\Big)
\ll F''(y).
\end{equation}
%Under our assumption \(y_0-k\geq k^{1/3+\epsilon}+10L\), for \(y\in[y_0-2L,y_0+2L]\) one has \(y-k\asymp y_0-k\gg k^{1/3+\epsilon}\), which implies
%\[\frac{\kappa^2}{y^2}(y^2-\kappa^2)^{-1/2}\ll \frac{\kappa^2}{y_0^{5/2}}(y_0-k)^{-1/2}\ll \Big(\frac{k}{y_0}\Big)^{3/2} y_0^{-1} k^{1/2} (y_0-k)^{-1/2} \ll y_0^{-1} k^{1/3-\epsilon/2}.\]
%On the other hand, using (\ref{y0exp}) (which gives \(y_0\asymp \sqrt{mx}/n\)) and the standing assumptions \(x\geq k^2/(8\pi^2)\) and \(m\leq \Delta^2k^\epsilon/x\), the assumption \(\Delta\leq x^{2/3}k^{1/3-\epsilon}\) of the lemma implies
%\[\frac{cn}{m}\asymp \frac{cx^{1/2}}{y_0 m^{1/2}}\gg \frac{x}{y_0 \Delta k^{\epsilon/2}}\gg y_0^{-1} x^{1/3}k^{-1/3+\epsilon/2}\gg y_0^{-1}\Big(\frac{x}{k^2}\Big)^{1/3}k^{1/3+\epsilon/2}\gg y_0^{-1}k^{1/3+\epsilon/2}.\]
%Combining the above two inequalities, we obtain \(\kappa^2 y^{-2} (y^2-\kappa^2)^{-1/2}=o(cn/m)\).
So from (\ref{tildeF''}) and (\ref{kappaF''bound}), we conclude 
\begin{equation}\label{tildeF''2}
F_\pm''(y)\asymp F''(y)\asymp  \frac{cn}{m}.
\end{equation}

Differentiating repeatedly, we obtain that for \(y\in[y_0-2L, y_0+2L]\) and \(j>2\), 
\begin{equation*}
F_\pm^{(j)}(y)=F^{(j)}(y)+\frac{d^{j-2}}{dy^{j-2}}\Big(\frac{\kappa^2}{y^2}(y^2-\kappa^2)^{-1/2}\Big)\ll_j \frac{cn}{m} y^{2-j} + \frac{k^2}{y^{2}}(y+k)^{-1/2}(y-k)^{-1/2-(j-2)}.
\end{equation*}
Here we used the bound (\ref{Fderivs}) for \(F^{(j)}(y)\). Since \(\kappa^2 y^{-2} (y^2-\kappa^2)^{-1/2}=o(cn/m)\) (see (\ref{kappaF''bound})), for \(y\in[y_0-2L, y_0+2L]\) (satisfying \(y-k\asymp y_0-k\)) we have
\begin{equation} \label{Fderivsbigthan2} 
F^{(j)}(y)\ll_j \frac{cn}{m} (y^{2-j} + (y-k)^{2-j})\ll_j \frac{cn}{m}(y_0-k)^{2-j}, \text{ for } j=2,3,4\ldots
\end{equation}

We also deduce that \(F_\pm\) has a single stationary point in the interval \([y_0-2L,y_0+2L]\). Indeed, \(F_\pm''\) does not change sign on this interval, since \( F_\pm''\asymp cn/m\) for all \(y\in[y_0-2L, y_0+2L]\) by (\ref{tildeF''2}). Thus \(F_\pm'\) is monotonic on this interval. Moreover, the mean value theorem shows
\begin{equation}\label{mvtF'} 
F'(y)=(y-y_0)F''(y_0+\xi), \text{ for some } \xi=\xi(y) \in\begin{cases} [0, y-y_0] & \text{ if } y\geq y_0,\\ [y-y_0, 0] & \text{ if } y\leq y_0.\end{cases}
\end{equation}
In particular, by (\ref{tildeF'})
\begin{multline*}
F'_\pm(y_0+2L)=F'(y_{0}+ 2L) +\mathcal O(1) = 2LF''(y_0+\xi_+)+\mathcal O(1)\\
\text{and } F'_\pm(y_0-2L)= -2LF''(y_0-\xi_-)+\mathcal O(1), \text{ for some \(0\leq \xi_+, \xi_-\leq 2L\)}.
\end{multline*}
Now since \(\pm 2LF''(y_0+\xi_\pm)\asymp cnL/m \gg k^\epsilon\) (by (\ref{tildeF''2}) and the fact \(L=k^\epsilon\max(1, m/(cn))\)), we deduce that \(F_\pm'\) has a unique zero in this interval, call this \( y_{0\pm}\).

Finally, if \(y\in[y_0-2L, y_0+2L]\), from (\ref{tildeF''2}) and (\ref{mvtF'}) we obtain the derivative bound
\begin{equation*}
F'(y)\ll |y-y_0||F''(y+\xi(y))|\ll \frac{cnL}{m}
\implies F_\pm'(y)=F'(y)+\mathcal O(1) \ll \frac{cnL}{m} \ll \frac{cn}{m} (y_0-k).
\end{equation*}
In the chain of inequalities on the right of the above, we used (\ref{tildeF'}) and our assumption \(y_0-k\geq 10L+k^{1/3+\epsilon}\gg L\). Thus we can extend (\ref{Fderivsbigthan2}) to the case \(j=1\) also. In other words 
\begin{equation}\label{tildeFhighderiv}
F_\pm^{(j)}(y)\ll_j \frac{cn}{m}(y_0-k)^{2-j}, \text{ valid for } j=1,2,\ldots
\end{equation}

Equipped with the estimates (\ref{tildegderivs}), (\ref{tildeF''2}) and (\ref{tildeFhighderiv}) and the fact that \(F_\pm'\) has a unique zero \(y_{0\pm}\) in \([y_0-2L, y_0+2L]\), we now apply Lemma \ref{bkyexpansion} with \(\alpha=y_0-2L\), \(\beta=y_0+2L\), \(X=x^{-3/4} y_0^{-3/4} (y_0-k)^{-1/4}\), \(V=L\), \(Y=cn(y_0-k)^2/m\) and \(Q=y_0-k\). 
%In this case, \(Z=Q+(\beta-\alpha)+X+Y+1\) clearly satisfies \(k^\epsilon\leq Z\leq k^{100}\), say, so to check the conditions (\ref{bkyexpansioncondition}) it suffices to find an \(\eta>0\) such that
To check the conditions (\ref{bkyexpansioncondition}) of Lemma \ref{bkyexpansion}, it suffices to show \(Y\geq k^\epsilon\) and \(VY^{1/2}/Q\geq k^\epsilon\).
Since we assume \(y_0-k\geq 10L+k^{1/3+\epsilon}\geq L\), we have
\[Y=\frac{cn}{m}(y_0-k)^2\geq \frac{cn}{m}L^2, \text{ and } \frac{VY^{1/2}}{Q}=\Big(\frac{cn}{m}\Big)^{1/2}L.\]
But \((cn/m)^{1/2}L=k^\epsilon \max\{(cn/m)^{1/2}, (cn/m)^{-1/2}\}\geq k^\epsilon\), so the conditions required in (\ref{bkyexpansioncondition}) are clearly satisfied.

%To check the conditions (\ref{bkyexpansioncondition}) of Lemma \ref{bkyexpansion}, it suffices to show 
%\[Y\geq k^\epsilon\; \text{ and }\; \frac{VY^{1/2}}{Q}\geq k^\epsilon.\]
%Since we assume \(y_0-k\geq 10L+k^{1/3+\epsilon}\geq L\), we have
%\[Y=\frac{cn}{m}(y_0-k)^2\geq \frac{cn}{m}L^2, \text{ and } \frac{VY^{1/2}}{Q}=\Big(\frac{cn}{m}\Big)^{1/2}L.\]
%Since \(L=k^\epsilon \max\{1, m/(cn)\}\), we have \((cn/m)^{1/2}L=k^\epsilon \max\{(cn/m)^{1/2}, (cn/m)^{-1/2}\}\geq k^\epsilon\), so the conditions required in (\ref{bkyexpansioncondition}) of Lemma \ref{bkyexpansion} are easily satisfied. 
We thus conclude the bound
\begin{equation*}
\int_{y_0-2L}^{y_0+2L} G_{i\pm} (y) e^{i F_\pm(y)}dy \ll x^{-3/4}y_0^{-3/4}(y_0-k)^{-1/4} c^{-1/2}m^{1/2}n^{-1/2}+k^{-1000}.
\end{equation*}
As previously, we have \(y_0\asymp \sqrt{mx}/n\) (see (\ref{y0exp})). Moreover, since we assume \(y_0-k\geq k^{1/3+\epsilon}\), it follows from (\ref{y0exp}) that \(y_0-k\asymp 4\pi\sqrt{mxt}/n-k\). Replacing these estimates in the above yields
%Recall (\ref{y0exp}) states \(y_0= 4\pi\sqrt{mxt}/n+\mathcal O(k^{1/3-\epsilon})\asymp \sqrt{mx}/n\). Since \(y_0-k\geq 10L+k^{1/3+\epsilon}\geq k^{1/3+\epsilon}\), it follows \(y_0-k\asymp 4\pi\sqrt{mxt}/n-k\). Therefore the above bound shows
%\begin{multline*} 
%\int_{y_0-2L}^{y_0+2L} G_{i\pm} (y) e^{i F_\pm(y)}dy \ll x^{-3/4}\Big(\frac{\sqrt{mx}}{n}\Big)^{-3/4}\Big(\frac{4\pi\sqrt{mxt}}{n}-k\Big)^{-1/4} c^{-1/2}m^{1/2}n^{-1/2} \\
%\ll x^{-9/8}k^{-1/4}c^{-1/2}m^{1/8}n^{1/2}\Big(\frac{4\pi\sqrt{mxt}}{k}-n\Big)^{-1/4}.
%\end{multline*}
\begin{equation*}
\int_{y_0-2L}^{y_0+2L} G_{i\pm} (y) e^{i F_\pm(y)}dy
\ll
x^{-9/8}k^{-1/4}c^{-1/2}m^{1/8}n^{1/2}\Big(\frac{4\pi\sqrt{mxt}}{k}-n\Big)^{-1/4}.
\end{equation*}
Part \ref{iiniii} now follows from (\ref{Ii'nexp4}).
\end{proof}

Finally, we apply the previous lemma to deduce the following bound for the sums \(S_i\).

\begin{lemma}\label{siboundfinal}
Assume \(\Delta\leq x^{2/3}k^{1/3-\epsilon}\). Then for \(i=1,2\) one has
\begin{multline*}
S_i\ll x^{-1/4}k^{-13/6-\epsilon}c^{1/2}m^{1/4}+x^{-1/2}k^{-3/2}m^{1/2}+x^{-3/4}k^{-3/2+2\epsilon}c^{-1/2}m^{3/4}\\
+x^{-3/4}k^{-1/2-2\epsilon}c^{1/2}m^{-1/4}+x^{-1}k^{-1/3}+x^{-5/4}k^{1/6+\epsilon}c^{-1/2}m^{1/4}+x^{-5/4}k^{-5/6+2\epsilon}c^{-3/2}m^{5/4}.
\end{multline*}
\end{lemma}

\begin{proof}
Recall from Lemma \ref{s12'lem} that 
\[S_i\ll c^{1/2}m^{-1/4}\sum_{\frac{m^{1/2}x^{3/2}}{\Delta^2k^\epsilon}\leq n\leq \frac{1000\sqrt{mx}}{k}} |\mathcal I_i'(n)|+k^{-1000}.\]
We now apply the bounds of Lemma \ref{blomerexplem}. To do so, we must split the sum over \(n\) into three parts based on the value of \(y_0(n)\) -- recall from (\ref{y0exp})
\begin{equation} \label{y0inequals} 
y_0=\frac{4\pi\sqrt{mxt}}{n}+\mathcal O(k^{1/3-\epsilon})\implies \frac{4\pi\sqrt{mxt}}{n}-k^{1/3}\leq y_0\leq \frac{4\pi\sqrt{mxt}}{n}+k^{1/3}.
\end{equation}

Firstly, we can see from part \ref{iini} of Lemma \ref{blomerexplem} that all terms with \(y_0+2L\leq k-k^{1/3+\epsilon}\) contribute \(\ll e^{- k^\epsilon/2}\), say, which is negligible. 
%(There are \(\leq 1000\sqrt{mx}/k\) of these terms, all of which are \(\ll c^{1/2}m^{-1/4} e^{-k^\epsilon}\).)

We next consider the contribution of \(\mathcal I_i'(n)\) where 
\begin{equation}\label{middlencond}
y_0(n)+2L\geq k-k^{1/3+\epsilon} \text{ and } y_0(n)-10L\leq k+k^{1/3+\epsilon}.
\end{equation} 
We first bound the number of \(n\) for which (\ref{middlencond}) holds. Since \(L\leq k^\epsilon(1+m/(cn))\), using (\ref{y0inequals}) we first observe 
\begin{multline*} 
y_0+2L\geq k-k^{1/3+\epsilon}\implies \frac{4\pi\sqrt{mxt}}{n}+k^{1/3}+2k^\epsilon\Big(1+\frac{m}{cn}\Big)\geq k-k^{1/3+\epsilon}\\
%\implies \frac{4\pi\sqrt{mxt}+2k^\epsilon m/c}{n}\geq k-2k^{1/3+\epsilon}
\implies n\leq \frac{4\pi\sqrt{mxt}+2k^\epsilon m/c}{k-2k^{1/3+\epsilon}}.
\end{multline*}
Similarly, 
\begin{equation*}
y_0-10L\leq k+k^{1/3+\epsilon}
%\implies \frac{4\pi\sqrt{mxt}}{n} -k^{1/3}-10k^\epsilon\Big(1+\frac{m}{cn}\Big)\leq k+k^{1/3+\epsilon}\\
%\implies \frac{4\pi\sqrt{mxt}-10k^\epsilon m/c}{n}\leq k+2k^{1/3+\epsilon}
\implies n\geq \frac{4\pi\sqrt{mxt}-10k^\epsilon m/c}{k+2k^{1/3+\epsilon}}.
\end{equation*}
Consequently, there are 
\[\leq \frac{4\pi\sqrt{mxt}+2k^\epsilon m/c}{k-2k^{1/3+\epsilon}}-\frac{4\pi\sqrt{mxt}-10k^\epsilon m/c}{k+2k^{1/3+\epsilon}}+1\ll x^{1/2}k^{-5/3+\epsilon}m^{1/2}+k^{-1+\epsilon}c^{-1}m+1\]
integers \(n\) satisfying (\ref{middlencond}). For these \(n\), we use the bound from part \ref{iinii} of Lemma \ref{blomerexplem}. This states 
\[\mathcal I_i'(n)\ll x^{-3/4}k^{-5/6+\epsilon}+x^{-5/4}k^{1/6+\epsilon}c^{-1}m^{1/2}.\]
We thus conclude that the overall contribution to \(S_i\) from \(n\) satisfying (\ref{middlencond}) is
\begin{multline}\label{middlencontribution} 
\ll c^{1/2}m^{-1/4}\big(x^{-3/4}k^{-5/6+\epsilon}+x^{-5/4}k^{1/6+\epsilon}c^{-1}m^{1/2}\big)\big(x^{1/2}k^{-5/3+\epsilon}m^{1/2}+k^{-1+\epsilon}c^{-1}m+1\big)\\
\ll x^{-1/4}k^{-5/2+2\epsilon}c^{1/2}m^{1/4}+x^{-3/4}k^{-11/6+2\epsilon}c^{-1/2}m^{3/4}+x^{-3/4}k^{-5/6+\epsilon}c^{1/2}m^{-1/4}\\
+x^{-3/4}k^{-3/2+2\epsilon}c^{-1/2}m^{3/4}+x^{-5/4}k^{-5/6+2\epsilon}c^{-3/2}m^{5/4}+x^{-5/4}k^{1/6+\epsilon}c^{-1/2}m^{1/4}.
\end{multline}
(Note the second term is dominated by the fourth term, so can be ignored.)

Finally, we bound the contribution of the remaining \(n\), for which \(y_0(n)-10L\geq k+k^{1/3+\epsilon}\). Note (using (\ref{y0inequals}))
\begin{equation*}
y_0-10L\geq k+k^{1/3+\epsilon}\implies \frac{4\pi\sqrt{mxt}}{n} +k^{1/3}\geq k+k^{1/3+\epsilon} \implies n\leq \frac{4\pi\sqrt{mxt}}{k+\frac12 k^{1/3+\epsilon}}.
\end{equation*}
For these \(n\), we use the bound given in part \ref{iiniii} of Lemma \ref{blomerexplem}. This states
\begin{equation*}
\mathcal I_i'(n)\ll x^{-9/8}k^{-1/4}c^{-1/2}m^{1/8}n^{1/2}\Big(\frac{4\pi\sqrt{mxt}}{k}-n\Big)^{-1/4}
+x^{-5/8}k^{-9/4}m^{1/8}n^{2}\Big(\frac{4\pi\sqrt{mxt}}{k}-n\Big)^{-9/4}.
\end{equation*}
It follows that the overall contribution to \(S_i\) from \(n\) satisfying \(y_0(n)-10L\geq k+k^{1/3+\epsilon}\) is 
\begin{multline}\label{largencontributions}
\ll c^{1/2}m^{-1/4}\Big(x^{-9/8}k^{-1/4}c^{-1/2}m^{1/8}\sum_{n\leq \frac{4\pi\sqrt{mxt}}{k+\frac12 k^{1/3+\epsilon}}}n^{1/2}\Big(\frac{4\pi\sqrt{mxt}}{k}-n\Big)^{-1/4}\\
+x^{-5/8}k^{-9/4}m^{1/8}\sum_{n\leq \frac{4\pi\sqrt{mxt}}{k+\frac12 k^{1/3+\epsilon}}} n^2\Big(\frac{4\pi\sqrt{mxt}}{k}-n\Big)^{-9/4}\Big).
\end{multline}
(We have dropped the restriction \(n\geq m^{1/2}x^{3/2}/(\Delta^2k^\epsilon)\) from these sums at no cost.) It is simple to bound\footnote{For an increasing function \(f\), one has \(\sum_{n\leq X}f(n)\leq \int_1^X f(u)du+f(X)\).} the remaining sums over \(n\). Firstly
\begin{multline*}
\sum_{n\leq \frac{4\pi\sqrt{mxt}}{k+\frac12 k^{1/3+\epsilon}}}n^{2}\Big(\frac{4\pi\sqrt{mxt}}{k}-n\Big)^{-9/4}\\
\leq \int_1^{\frac{4\pi\sqrt{mxt}}{k+\frac12 k^{1/3+\epsilon}}} u^{2}\Big(\frac{4\pi\sqrt{mxt}}{k}-u\Big)^{-9/4}du+\Big(\frac{4\pi\sqrt{mxt}}{k+\frac12 k^{1/3+\epsilon}}\Big)^{2}\Big(\frac{4\pi \sqrt{mxt}}{k}-\frac{4\pi\sqrt{mxt}}{k+\frac12 k^{1/3+\epsilon}}\Big)^{-9/4}\\
%\ll \Big(\frac{\sqrt{mx}}{k}\Big)^2\Big(\frac{\sqrt{mx}}{k^{5/3-\epsilon}}\Big)^{-5/4}+\Big(\frac{\sqrt{mx}}{k}\Big)^2\Big(\frac{\sqrt{mx}}{k^{5/3-\epsilon}}\Big)^{-9/4}\\
\ll x^{3/8}k^{1/12-5\epsilon/4}m^{3/8}+x^{-1/8}k^{7/4-9\epsilon/4}m^{-1/8}.
\end{multline*}
Similarly
\begin{equation*}
\sum_{n\leq \frac{4\pi\sqrt{mxt}}{k+\frac12 k^{1/3+\epsilon}}}n^{1/2}\Big(\frac{4\pi\sqrt{mxt}}{k}-n\Big)^{-1/4}
\ll 
x^{5/8}k^{-5/4}m^{5/8}+x^{1/8}k^{-1/12-\epsilon/4}m^{1/8}.
\end{equation*}
%\begin{multline*}
%\sum_{n\leq \frac{4\pi\sqrt{mxt}}{k+\frac12 k^{1/3+\epsilon}}}n^{1/2}\Big(\frac{4\pi\sqrt{mxt}}{k}-n\Big)^{-1/4}\\
%\leq \int_1^{\frac{4\pi\sqrt{mxt}}{k+\frac12 k^{1/3+\epsilon}}} u^{1/2}\Big(\frac{4\pi\sqrt{mxt}}{k}-u\Big)^{-1/4}du+\Big(\frac{4\pi\sqrt{mxt}}{k+\frac12 k^{1/3+\epsilon}}\Big)^{1/2}\Big(\frac{4\pi \sqrt{mxt}}{k}-\frac{4\pi\sqrt{mxt}}{k+\frac12 k^{1/3+\epsilon}}\Big)^{-1/4}\\
%\ll \int_{\frac{4\pi\sqrt{mxt}}{k}-\frac{4\pi\sqrt{mxt}}{k+\frac12 k^{1/3+\epsilon}}}^{\frac{4\pi\sqrt{mxt}}{k}-1} \Big(\frac{4\pi\sqrt{mxt}}{k}+v\Big)^{1/2} v^{-1/4}dv+x^{1/8}k^{-1/12-\epsilon}m^{1/8}\\
%\ll x^{5/8}k^{-5/4}m^{5/8}+x^{1/8}k^{-1/12-\epsilon}m^{1/8}.
%\end{multline*}

Consequently, the contribution of all \(n\) with \(y_0(n)-10L\geq k+k^{1/3+\epsilon}\), given in (\ref{largencontributions}) is 
\begin{equation}\label{largencontrib2} 
\ll x^{-1/2}k^{-3/2}m^{1/2}+x^{-1}k^{-1/3}+x^{-1/4}k^{-13/6-\epsilon}c^{1/2}m^{1/4}+x^{-3/4}k^{-1/2-2\epsilon}c^{1/2}m^{-1/4}.
\end{equation}
Combining the bounds (\ref{middlencontribution}) and (\ref{largencontrib2}) for the two (non-negligible) contributions to the sums \(S_i\) (noting also that the first and third terms of (\ref{middlencontribution}) are dominated by the third and fourth terms of (\ref{largencontrib2}) respectively, so can be ignored), we obtain the lemma.
\end{proof}

Equipped with this bound for the sums \(S_i\), the main result of this section (Lemma \ref{offdiaglem}) follows straightforwardly from Lemma \ref{initalodbound}.

\begin{proof}[Proof of Lemma \ref{offdiaglem}]
Recall Lemma \ref{initalodbound}, which states 
\begin{equation*}
(\mathrm{OD})\ll x^{5/4}\sum_{c\leq 100\Delta^2/(x k^{1-\epsilon})} c^{-1}\sum_{m\leq \Delta^2k^\epsilon/x} m^{-3/4} \big(\max_{t\in[1,2]}|S_1|+\max_{t\in[1,2]}|S_2|\big)+x^{-5/4}k^{-4/3+2\epsilon}\Delta^{5/2}.
\end{equation*}
Lemma \ref{siboundfinal} (which holds uniformly for \(t\in[1,2]\)) now yields 
\begin{multline*}
(\mathrm{OD})\ll  x k^{-13/6-\epsilon}\sum_c c^{-1/2}\sum_m m^{-1/2}+ x^{3/4}k^{-3/2}\sum_c c^{-1}\sum_m m^{-1/4}\\
x^{1/2} k^{-3/2+2\epsilon} \sum_c c^{-3/2}\sum_m 1+ x^{1/2} k^{-1/2-2\epsilon} \sum_c c^{-1/2} \sum_m m^{-1}+x^{1/4}k^{-1/3}\sum_c c^{-1} \sum_m m^{-3/4}\\
k^{1/6+\epsilon} \sum_c c^{-3/2} \sum_m m^{-1/2}
+ k^{-5/6+ 2\epsilon} \sum_c c^{-5/2} \sum_m m^{1/2}
+x^{-5/4}k^{-4/3+2\epsilon}\Delta^{5/2},
\end{multline*}
where \(\sum_c\) denotes that the sum is taken over \(c\leq 100\Delta^2/(x k^{1-\epsilon})\), and \(\sum_m\) denotes that the sum is taken over \(m\leq \Delta^2 k^\epsilon/x\). One easily bounds these sums, and obtains
\begin{multline*}%\label{odboundeventual}
(\mathrm{OD}) \ll k^{-8/3}\Delta^2+k^{-3/2+\epsilon}\Delta^{3/2}+x^{-1/2}k^{-3/2+3\epsilon}\Delta^2\\
+k^{-1-\epsilon}\Delta +k^{-1/3+\epsilon}\Delta^{1/2}+ x^{-1/2} k^{1/6+2\epsilon} \Delta+ x^{-3/2}k^{-5/6+4\epsilon}\Delta^3+x^{-5/4}k^{-4/3+2\epsilon}\Delta^{5/2}.
\end{multline*}
Our assumptions \(x\geq k^2/(8\pi^2)\) and \(x^{1/2}\leq \Delta\leq x^{2/3}k^{1/3-\epsilon}\) imply that the third and fourth error term above are dominated by the second, and that the final error term is dominated by the penultimate one.
%We may simplify this expression somewhat, using our assumptions \(x\geq k^2/(8\pi^2)\) and \(x^{1/2}\leq \Delta\leq x^{2/3}k^{1/3-\epsilon}\). These imply
%\begin{multline*} 
%x^{-1/2}k^{-3/2+3\epsilon}\Delta^2\ll x^{-1/6}k^{-4/3+5\epsilon/2}\Delta^{3/2}= k^{-3/2+\epsilon}\Delta^{3/2}\cdot x^{-1/6}k^{1/6+3\epsilon/2} \ll k^{-3/2+\epsilon}\Delta^{3/2}\\
%\text{ and } k^{-1-\epsilon}\Delta\leq x^{-1/4}k^{-1-\epsilon}\Delta^{3/2}\ll k^{-3/2+\epsilon}\Delta^{3/2}.
%\end{multline*}
%Therefore both the third and fourth error term in (\ref{odboundeventual}) can be absorbed into the second. We also have
%\[x^{-5/4}k^{-4/3+2\epsilon}\Delta ^{5/2} \leq x^{-3/2}k^{-4/3+2\epsilon}\Delta^3\leq x^{-3/2}k^{-5/6+4\epsilon}\Delta^3,\]
%so the final error term in (\ref{odboundeventual}) can be absorbed into the penultimate one. 
We thus obtain the required bound
\[(\mathrm{OD})\ll k^{-8/3}\Delta^2+k^{-3/2+\epsilon}\Delta^{3/2}+k^{-1/3+\epsilon}\Delta^{1/2}+x^{-1/2} k^{1/6+2\epsilon} \Delta+x^{-3/2}k^{-5/6+4\epsilon}\Delta^3.\]
\end{proof}

\bibliographystyle{plain}
\bibliography{References}

\end{document}